%% file: main.tex
\documentclass[12pt,a4paper]{amsart}

\usepackage[a4paper,hmargin=2.5cm,vmargin=2.5cm]{geometry}
\usepackage[utf8]{inputenc}
\usepackage[T1]{fontenc}
\usepackage{graphicx}
\usepackage{indentfirst}
\usepackage{amsfonts,amsmath,amssymb,amsthm}
\usepackage{tikz-cd}
\usepackage{xcolor}
\usepackage[main=english,russian]{babel}

\usepackage{microtype}

\usepackage[leqno]{amsmath}
\makeatletter
\newcommand{\reqnomode}{\tagsleft@false}

\makeatother

\usepackage{xcolor}
\usepackage{hyperref}

\hypersetup{
    colorlinks=true,
    linkcolor=black,
    citecolor=black,
    pdftitle={Yangian for gl_2 and Gelfand--Tsetlin basis for gl_n},
    pdfauthor={Mark Belov},
    pdfkeywords={Gelfand-Tsetlin, Yangians, Lie algebras}
    pdfpagemode=UseNone
}

\usepackage{comment}

\newtheorem{theorem}{Theorem}
\numberwithin{theorem}{section}

\newtheorem*{theorem*}{Theorem}

\theoremstyle{definition} 
\newtheorem{definition}[theorem]{Definition}

\theoremstyle{proposition} 
\newtheorem{proposition}[theorem]{Proposition}

\theoremstyle{proposition} 
\newtheorem{remark}[theorem]{Remark}

\theoremstyle{corollary} 
\newtheorem{corollary}[theorem]{Corollary}

\input{formulas}

\begin{document}
\reqnomode

\numberwithin{equation}{section}

\title{Yangian for $\gl{2}$ and Gelfand--Tsetlin basis for $\gl{n}$}
\begin{abstract}
    We construct an explicit realization of finite-dimensional irreducible representations $V_{\lambda}$ of $\gl{n}$ using the representation theory of $\yang{2}$. We consider $V_{\lambda}$ as the evaluation representation of $\yang{n}$ and restrict it to the subalgebras $\yang{2}^{(i)}$. Then $V_{\lambda}$ is completely reducible as a representation of $\yang{2}^{(i)}$. This allows us to use the explicit formulas for irreducible representations of $\yang{2}$.
\end{abstract}

\author{Mark Belov}
\address{Faculty of Mathematics, HSE University, 6 Usacheva St., Moscow 119048, Russia}
\email{maanbelov@edu.hse.ru}

\maketitle

\section{Introduction}
\label{Section1}
\input{1_Intro}

\section{The Gelfand-Tsetlin basis for \texorpdfstring{$\gl{n}$}{gl n}}
\label{Section2}
\input{2_Definition_of_GZ_Basis}

\section{Yangian and its representations}
\label{Section3}
\input{3_Yangians}

\section{The Gelfand-Tsetlin basis for \texorpdfstring{$\yang{2}$}{Y(gl 2)}}
\label{Section4}
\input{4_GZ_Basis_for_Y_gl2}

\section{The Gelfand-Tsetlin basis for \texorpdfstring{$\gl{n}$}{gl n}}
\label{Section5}
\input{5_GZ_Basis_for_gln}

\end{document}

%% file: formulas.tex
\newcommand{\gl}[1]{\ensuremath{\mathfrak{gl}_{#1}}}

\newcommand{\yang}[1]{\ensuremath{\mathrm{Y}(\mathfrak{gl}_{#1})}}

\newcommand{\uni}[1]{\ensuremath{\mathrm{U}(\mathfrak{gl}_{#1})}}

\newcommand{\Lir}[2]{\ensuremath{\mathit{L}(#1, #2)}}

\newcommand{\gz}[1]{\ensuremath{\mathrm{GZ}(#1)}}

\newcommand{\End}[1]{\mathrm{End}~\mathbb{C}^{#1}}

\newcommand{\suminf}[2]{\ensuremath{\sum_{#1=#2}^{+\infty}}}

\renewcommand{\t}[3]{\ensuremath{t_{#1#2}^{(#3)}}}

\newcommand{\qdet}[1]{\ensuremath{\mathrm{qdet}~T(#1)}}

\newcommand{\sgn}[1]{\ensuremath{\mathrm{sgn}(#1)\,}}

%% file: 1_Intro.tex
\subsection{Representations of \texorpdfstring{\gl{n}}{gl n} and Gelfand-Tsetlin basis}
We consider the Lie algebra of all complex $n\times n$-matrices $\gl{n}$ and its universal enveloping algebra $\uni{n}$. Let $E_{ij}$ for $i,j = 1, \ldots, n$ be standard generators of $\gl{n}$, that is, matrix unit with $1$ on $(i,j)$-entry and $0$ elsewhere.

Finite-dimensional irreducible representations $V_{\lambda}$ of $\gl{n}$ are parametrized by $n$-tuples
\begin{align*}
    \lambda=(\lambda_1,\ldots,\lambda_n) \in \mathbb{C}^n
\end{align*} such that $\lambda_i - \lambda_j \in \mathbb{Z}_{\geq 0}$. 
The standard proofs of this result (see e.g. \cite{Fulton}) do not provide formulas of the action of $\gl{n}$ on $V_{\lambda}$. Explicit formulas were first given by I. M. Gelfand and M. L. Tsetlin in the famous paper \cite{OriginalGelfand}. They use  a Gelfand-Tsetlin basis which is parametrized by Gelfand-Tsetlin tableaux, see \ref{tabl}. However, the paper contains only the formulas, with no methods for obtaining them. 

There are several approaches to deriving the Gelfand-Tsetlin formulas, see \cite{MolevReview} and the references given there. In particular, there is a method in the literature that uses the theory of Yangians: 
\begin{enumerate}
    \item In Nazarov and Tarasov's paper \cite{NazarovTarasov}, they consider lowering operators in the universal enveloping algebra $\uni{n}$ through $\yang{n}$;
    \item In the paper \cite{MolevFirst} by Molev, the ideas of the paper \cite{NazarovTarasov} were developed by constructing the Gelfand-Tsetlin bases for some irreducible representations of $\yang{n}$, see also~\cite{NazarovTarasov2}.
\end{enumerate}

In the present paper we use the following approach: we restrict the evaluation representation $V_{\lambda}$ to the subalgebras $\yang{2}^{(i)} \simeq \yang{2} \subset \yang{n}$, decompose the restriction as a direct sum of irreducible $\yang{2}$-modules and use a variant of Gelfand-Tsetlin formulas for representations of $\yang{2}^{(i)}$. Hence, in a sense, we present a different construction of the Gelfand-Tsetlin formulas via the theory of Yangians.

\subsection{Yangian for \texorpdfstring{$\gl{n}$}{gl n}}
The Yangian of the Lie algebra $\gl{n}$, denoted by $\yang{n}$, is a unital associative algebra generated by elements $\t{i}{j}{k}$ with $i,j\in\{1,\ldots,n\}$ and $k\in \mathbb{Z}_{\geq 0}$ subject to the relations
    \begin{align*}
        [\t{i}{j}{r+1}, \t{k}{l}{s}] - [\t{i}{j}{r}, \t{k}{l}{s+1}] = \t{k}{j}{r}\t{i}{l}{s} - \t{k}{j}{s}\t{i}{l}{r}
    \end{align*} for all $i,j\in\{1,\ldots,n\}$ and $r,s\in\mathbb{Z}_{\geq 0}$ with the assumption $\t{i}{j}{0}=\delta_{ij}$.

There is a surjective homomorphism
\begin{align*}
    \pi:\yang{n}\to\uni{n}: t_{ij}^{(k)}\mapsto \delta_{k1}E_{ij},
\end{align*} which is known as the evaluation homomorphism. It allows  to consider any representation of the Lie algebra $\gl{n}$ as a representation of $\yang{n}$.

$\yang{2}$ can be naturally embedded in $\yang{n}$ for all distinct $i,j\in\{1,\ldots,n\}$. However, in this paper we use different embeddings $\yang{2} \hookrightarrow \yang{n}$ for $i \in \{1,\ldots,n-1\}$ that arise from the Drinfeld new realization of $\yang{n}$, see Theorem \ref{new_realization}. We denote the image by $\yang{2}^{(i)}$.

\subsection{Main result} 
The main idea of the present paper is to restrict an irreducible representation $V_{\lambda}$ of $\gl{n}$ which we consider as a representation of $\yang{n}$ to the subalgebras $\yang{2}^{(i)}$.
It turns out that $V_{\lambda}\big|_{\yang{2}^{(i)}}$ is completely reducible and can be effectively decomposed into a direct sum of irreducible representations. This allows us to use an explicit realization of the corresponding irreducible representations of $\yang{2}$.

A Gelfand-Tsetlin tableau is an array of complex numbers 
    \begin{center}\begin{tabular}{c }
         $\lambda_{n1}$ ~  $\lambda_{n2}$ ~ \dots ~ $\lambda_{n,n-1}$ ~$\lambda_{nn}$ \\
         $\lambda_{n-1,1}$ ~ \dots ~ $\lambda_{n-1,n-1}$ \\
         $\ddots$ ~ $\dots$ ~\reflectbox{$\ddots$} \\
         $\lambda_{21}$ ~ $\lambda_{22}$ \\
         $\lambda_{11}$
    \end{tabular}
    \end{center}
such that $\lambda_{ij} - \lambda_{i-1,j} \in \mathbb{Z}_{\geq 0}$ and $\lambda_{i,j}- \lambda_{i,j+1} \in \mathbb{Z}_{\geq 0}$
for all possible indexes $i\geq j$. We now formulate the main theorem.

\begin{theorem*}
    Any finite-dimensional representation $V_{\lambda}$ of $\gl{n}$ of the highest weight $\lambda$ has the basis $\{v_{\Lambda}\}$, indexed by Gelfand-Tsetlin tableaux $\Lambda$ with the first row equals $\lambda$ such that 
    \begin{gather*}
        E_{ii}\cdot v_{\Lambda} = \left(\sum_{j=1}^{i}\lambda_{ij}-\sum_{j=1}^{i-1}\lambda_{i-1,j}\right)v_{\Lambda},\\
        E_{i+1,i}\cdot v_{\Lambda} = \sum_{j=1}^{i}\frac{1}{(l_{i1}-l_{ij})\cdot\ldots\land^{j}\ldots\cdot(l_{ii}-l_{ij})}v_{\Lambda-\delta_{ij}},\\
        E_{i,i+1}\cdot v_{\Lambda} = -\sum_{j=1}^{i} \frac{\prod_{k=1}^{i+1}(l_{i+1,k}-l_{ij}) \cdot \prod_{k=1}^{i-1}(l_{i-1,k}-l_{ij}-1)}{(l_{i1}-l_{ij})\cdot\ldots\land^{j}\ldots\cdot(l_{ii}-l_{ij})}v_{\Lambda+\delta_{ij}}.
    \end{gather*}
\end{theorem*}

\subsection{The paper is organized as follows.} In \hyperref[Section2]{Section 2}, we introduce the Gelfand-Tsetlin basis for $\gl{n}$. In \hyperref[Section3]{Section 3}, we provide general knowledge of Yangian theory from \cite{Molev} without proofs to make the exposition self-contained. \hyperref[Section4]{Section 4} contains the Gelfand-Tsetlin realization of the special class of irreducible $\yang{2}$-modules. Finally, in \hyperref[Section5]{Section 5}, the main result is proved.

\subsection{Acknowledgments}
The author expresses his sincere gratitude to his research supervisor, Aleksei Ilin, for valuable advice, constructive feedback, and careful reading of the manuscript.

%% file: 2_Definition_of_GZ_Basis.tex
\subsection{The basis}
Consider a chain of embeddings
\begin{align}
\label{chain_of_gln}
    \gl{1} \hookrightarrow \gl{2} \hookrightarrow \ldots \hookrightarrow \gl{n-1} \hookrightarrow \gl{n} \hookrightarrow \ldots,
\end{align} where $\gl{n-1}$ is embedded in $\gl{n}$ in the following rule:
\begin{align*}
    A \mapsto \begin{pmatrix}
        ~ & ~ & ~ & 0 \\
        ~ & A & ~ & \vdots \\
        ~ & ~ & ~ & 0 \\
        0 & \ldots & 0 & 0
    \end{pmatrix} \in\gl{n}, ~A\in\gl{n-1}.
\end{align*} We will identify $\gl{n-1}$ with its image under the injection if necessary.

\begin{definition}
    A representation $V$ of $\gl{n}$ is said to be a representation of the highest weight $\lambda=(\lambda_1, \ldots, \lambda_n)\in\mathbb{C}^n$ if there exists a vector $\zeta$, called the highest vector, such that 
    \begin{align*}
        E_{ii}\zeta=\lambda_i\zeta, ~E_{ij}\zeta=0 ~\text{if}~i<j ~~\text{and}~~\gl{n}\cdot \zeta=V.
    \end{align*}
\end{definition}
The next theorem is classical (see e.g. \cite{Fulton}), thus, we omit the proof.
\begin{theorem}
    \label{gl_n_classification_theorem}
    Every finite-dimensional irreducible representation of $\gl{n}$ is a representation of the highest weight $\lambda=(\lambda_1, \ldots, \lambda_n)$ satisfying
    \begin{align*}
        \lambda_i-\lambda_j\in\mathbb{Z}_{\geq 0} ~\text{for}~ 1 \leq i < j \leq n.
    \end{align*} Conversely, for every highest weight as above there exists a unique finite-dimensional irreducible representation of this weight. We will denote such a representation by $V_{\lambda}$.
\end{theorem}

Recall that for every fixed $c\in\mathbb{C}$ a map
\begin{align*}
    \uni{n}\to\uni{n}:E_{ij}\mapsto E_{ij}+\delta_{ij}c
\end{align*} is an automorphism of $\uni{n}$. After applying the map, the highest weight of irreducible representation $V_{\lambda}$ is shifted by $c$. Therefore, without loss of generality we can assume that
\begin{align*}
    \lambda_i \in \mathbb{Z}, ~\lambda_{1} \geq \ldots \geq\lambda_n \geq 0,
\end{align*} in that case $\lambda$ can be seen as a Young tableau with $n$ rows.

It is a well-known fact that the restriction of $V_{\lambda}$ to $\gl{n-1}$ is a direct sum of irreducible representations without multiplicity:
\begin{align*}
    V_{\lambda}\big|_{\gl{n-1}} = \bigoplus_{\mu} V_{\mu},
\end{align*} where $V_{\mu}$ is an irreducible representation $\gl{n-1}$ of the highest weight $\mu$ satisfied a condition
\begin{align}
\label{young_tableau_relation}
    \lambda_1 \geq \mu_1 \geq \lambda_2 \geq \ldots \geq \lambda_{n-1} \geq \mu_{n-1} \geq \lambda_n \geq 0.
\end{align}
The standard proof is based on the Weyl Character Formula and some relations of Schur polynomials, see \cite[6 sec.]{Fulton} for more details. Alternative proof can be found in \cite[13 par.]{Zhelob}.

Continuing the process, we will represent $V_{\lambda}$ as a direct sum of $\gl{1}$-irreducible modules, that are one-dimensional subspaces. Since the decomposition is without multiplicity at each step, we conclude that the received lines are in one-to-one correspondence with a chain of embedded Young tableaux
\begin{align}
\label{chain_of_young_tableau}
    \lambda^{(1)} \subset \lambda^{(2)} \subset \ldots \subset \lambda^{(n-1)} \subset \lambda^{(n)} = \lambda,
\end{align} with $\lambda^{(i)}$ consisting of $i$ rows and satisfying with $\lambda^{(i+1)}$ the similar relations as \eqref{young_tableau_relation}.

For simplicity, we introduce the notion of a Gelfand-Tsetlin tableau.

\begin{definition}
\label{tabl}
    The Gelfand-Tsetlin tableau $\Lambda$ for the Young diagram $\lambda$ is defined to be a set of integer non-negative numbers
    \begin{align*}
        \lambda_{ij} \in \mathbb{Z}_{\geq 0}, ~\text{where}~ 1 \leq j \leq i \leq n,
    \end{align*} which are satisfied the relations
    \begin{align*}
        \lambda_{ij} \geq \lambda_{i-1,j} \geq \lambda_{i,j+1}
    \end{align*} for every allowed indices $i$ and $j$ and
    \begin{align*}
        (\lambda_{n1},\ldots,\lambda_{nn}) = \lambda.
    \end{align*} The Gelfand-Tsetlin tableau is usually written in a way
    \begin{center}\begin{tabular}{c }
         $\lambda_{n1}$ ~  $\lambda_{n2}$ ~ \dots ~ $\lambda_{n,n-1}$ ~$\lambda_{nn}$ \\
         $\lambda_{n-1,1}$ ~ \dots ~ $\lambda_{n-1,n-1}$ \\
         $\ddots$ ~ $\dots$ ~\reflectbox{$\ddots$} \\
         $\lambda_{21}$ ~ $\lambda_{22}$ \\
         $\lambda_{11}$
    \end{tabular}
    \end{center}
\end{definition}

The chain of Young diagrams \eqref{chain_of_young_tableau} uniquely corresponds with a Gelfand-Tsetlin tableau $\Lambda$ with $i$-th row equals to $\lambda^{(i)}$.

Hence, every irreducible representation $V_{\lambda}$ contains a set of one-dimensional subspaces indexed by Gelfand-Tsetlin tableaux $\Lambda$. With the notation of the vector by $v_{\Lambda}$, we have an arbitrary non-zero vector on the corresponding line unless otherwise specified. The set of vectors $\{v_{\Lambda}\}$ is a basis of $V_{\Lambda}$ called the Gelfand-Tsetlin basis.

\subsection{The Gelfand-Tsetlin subalgebra}The Gelfand-Tsetlin basis has another characterization. The injections \eqref{chain_of_gln} induce a chain of embeddings
\begin{align}
\label{chain_of_unin}
    \uni{1} \hookrightarrow \uni{2} \hookrightarrow \ldots \hookrightarrow \uni{n-1} \hookrightarrow \uni{n} \hookrightarrow \ldots.
\end{align} Henceforth, we will identify $\uni{n-1}$ with its image of the injection.

\begin{definition}
\label{gelfand_tsetlin_subalgebra}
    The Gelfand-Tsetlin subalgebra $\gz{n}$ is a commutative subalgebra of $\uni{n}$ generated by corresponding centers of $\uni{i}$:
    \begin{align*}
        \gz{n} = \left<Z\uni{1}, \ldots, Z\uni{n}\right>.
    \end{align*}
\end{definition}

\begin{theorem}
\label{gelfand_tsetlin_theorem}
    The Gelfand-Tsetlin subalgebra $\gz{n}$ acts by diagonal matrices in the Gelfand-Tsetlin basis in every finite-dimensional irreducible representation. Moreover, any basis vector $v_{\Lambda}$ is uniquely defined up to scalar by the $\gz{n}$ acting.
\end{theorem}

\begin{proof}
    By Schur's Lemma, every element of the center $Z\uni{i}$ acts by a scalar in any finite-dimensional irreducible $\gl{i}$-module. By the definition of the Gelfand-Tsetlin basis, every vector $v_{\Lambda}$ is in some irreducible $\gl{i}$-module for each $i\in\{1,\ldots,n\}$ and thus is an eigenvector for $Z\uni{i}$. Since $\gz{n}$ is generated by $Z\uni{i}$, it proves the first statement.

    Indeed, the second statement will be proved straightforwardly later, see Proposition \ref{h_i_defines_GZ_basis}. However, it may be proved by a more general statement well-known as the Harish-Chandra isomorphism. It states that $Z\uni{n}$, as a commutative algebra, is isomorphic to the space of symmetric polynomials in the variables $\lambda_i+n-i$. In consequence, the acting of $Z\uni{n}$ uniquely defines the finite-dimensional irreducible representation, see \cite[\MakeUppercase{\romannumeral 9} ch.]{ZhelobCompact} for details.

    Since the $i$-th row in a Gelfand-Tsetlin tableau is a highest weight of $\gl{i}$-module in which corresponding $v_{\Lambda}$ is, the action of $Z\uni{i}$ defined the $i$-th row. Thus, any basis vector $v_{\lambda}$ is uniquely determined by the $\gz{n}$ acting.
\end{proof}

\begin{corollary}
    The image of $\gz{n}$ in any $\mathrm{End}(V_{\lambda})$ is all diagonal matrices.
\end{corollary}

\begin{proof}
    Fix an arbitrary basis vector $v_{\Lambda}$. For any distinct basis vector $v_{\hat{\Lambda}}$ one can find an element $g_{\hat{\Lambda}}\in\gz{n}$ with the property that its eigenvalues for $v_{\Lambda}$ and $v_{\hat{\Lambda}}$ differ.
    
    Let $C_{\hat{\Lambda}}\in\mathbb{C}$ be the eigenvalue of $g_{\hat{\Lambda}}\in\gz{n}$, consider 
    \begin{align*}
        g=\prod_{\hat{\Lambda} \not=\Lambda}\, (g_{\hat{\Lambda}}-C_{\hat{\Lambda}}\cdot 1) \in \gz{n}.
    \end{align*} It is well defined since $V_{\lambda}$ is finite-dimensional. Also, $g$ satisfies
    \begin{align*}
        g = C\cdot \mathrm{pr}_{v_{\Lambda}}
    \end{align*} for some non-zero $C\in\mathbb{C}\setminus\{0\}$. It proves the assertion.
\end{proof}

In fact, the more general statement is true. We refer the reader to \cite{Ovsienko} and \cite{Tarasov} for the proof and discussion the next remark.

\begin{remark}
    The Gelfand-Tsetlin subalgebra $\gz{n}$ is a maximal commutative subalgebra of $\uni{n}$.
\end{remark}

%% file: 3_Yangians.tex
\subsection{Preliminaries on Yangian}
We state some definitions and results from \cite{Molev}.

\begin{definition}
    The Yangian of the Lie algebra $\gl{n}$, denoted by $\yang{n}$, is a unital associative algebra generated by elements $\t{i}{j}{k}$ with $i,j\in\{1,\ldots,n\}$ and $k\in\mathbb{Z}_{\geq 0}$ subject to the relations
   \begin{align}
    \label{t_ijk_relations}
        [\t{i}{j}{r+1}, \t{k}{l}{s}] - [\t{i}{j}{r}, \t{k}{l}{s+1}] = \t{k}{j}{r}\t{i}{l}{s} - \t{k}{j}{s}\t{i}{l}{r}
    \end{align} for all $i,j\in\{1,\ldots,n\}$ and $r,s\in\mathbb{Z}_{\geq 0}$ with the assumption $\t{i}{j}{0}=\delta_{ij}$.
\end{definition}

Consider the next formal series of $\yang{n}[[u^{-1}]]$
\begin{align*}
    t_{ij}(u) = \delta_{ij} + \t{i}{j}{1}u^{-1} + \t{i}{j}{2}u^{-2} + \ldots = \suminf{k}{0}\t{i}{j}{k}u^{-k}.
\end{align*} 
Relations  \eqref{t_ijk_relations} are  equivalent to the relations
\begin{align}
\label{t_iju_relations}
    (u-v)[t_{ij}(u), t_{kl}(v)] = t_{kj}(u)t_{il}(v) - t_{kj}(v)t_{il}(u).
\end{align}

\begin{proposition}
\label{evaluation_homomorphism}
    A map $\pi:\yang{n}\to\uni{n}$ defined by 
    \begin{align*}
        t_{ij}(u) \to \delta_{ij} + E_{ij}u^{-1}
    \end{align*} is a surjective algebra homomorphism.
\end{proposition}

\begin{proof}
    Writing relations \eqref{t_ijk_relations} with $r=0$ and $s=1$, we get
    \begin{align*}
        [\t{i}{j}{1},\t{k}{l}{1}] = \delta_{kj}\t{i}{j}{1} - \delta_{il}\t{k}{j}{1}.
    \end{align*} Since the next relations hold in $\uni{n}$
    \begin{align*}
        [E_{ij}, E_{kl}] = \delta_{kj}E_{il} - \delta_{il}E_{kj},
    \end{align*} the map $\pi$ is an algebra homomorphism. Clearly, it is a surjection.
\end{proof}
The map $\pi$ is called the evaluation homomorphism. 

Recall that we fixed the chain of embeddings \eqref{chain_of_unin} of $\uni{n}$. In the same manner we consider a chain of embeddings of $\yang{n}$
\begin{align}
\label{yangian_embeddings}
    \yang{1} \hookrightarrow \yang{2}\hookrightarrow\ldots\hookrightarrow\yang{n-1}\hookrightarrow\yang{n}\hookrightarrow\ldots,
\end{align} where the injection $\yang{n-1} \hookrightarrow\yang{n}$ is naturally given by
\begin{align*}
    \t{i}{j}{k} \to \t{i}{j}{k},~ \forall 1 \leq i,\,j \leq n-1.
\end{align*} PBW theorem for $\yang{n}$ (see \cite[Th. 1.4.1]{Molev}) guarantees that the maps are truly injective. These embeddings are agreed with the evaluation homomorphism, that is, the following diagram is commutative
\begin{center}
    \begin{tikzcd}
        \yang{1} \arrow[hook]{r} \arrow[two heads]{d}{\pi_{1}} & \yang{2} \arrow[r, hook]\arrow[two heads]{d}{\pi_{2}} & \ldots \arrow[r, hook]& \yang{n-1}\arrow[r, hook] \arrow[two heads]{d}{\pi_{n-1}} & \yang{n} \arrow[two heads]{d}{\pi_{n}} \arrow[hook]{r} &\ldots\\
        \uni{1} \arrow[r, hook] & \uni{2} \arrow[r, hook] & \ldots\arrow[r, hook] & \arrow[r] \uni{n-1} \arrow[r, hook] & \uni{n} \arrow[hook]{r} &\ldots
    \end{tikzcd}
\end{center}

Set $T(u)$ as a matrix with elements of $\yang{n}[[u^{-1}]]$ where $ij$-th entry equals $t_{ij}(u)$:
\begin{align}
\label{TU_matrix}
    T(u) = [t_{ij}(u)]_{ij} = \begin{pmatrix}
        t_{11}(u) & t_{12}(u) & \ldots & t_{1n}(u) \\
        t_{21}(u) & t_{22}(u) & \ldots & t_{2n}(u) \\
        \vdots & \vdots & \ddots & \vdots \\
        t_{n1}(u) & t_{n2}(u) & \ddots & t_{nn}(u)
    \end{pmatrix}.
\end{align} The matrix $T(u)$ can be regarded as an element from $\End{n}\otimes\yang{n}[[u^{-1}]]$.

After fixing $m\in\mathbb{N}$, we denote by $T_a(u)$ for $a\in\mathbb{N}$ with $a\leq m$ the operator from $(\End{n})^{\otimes m}\otimes\yang{n}[[u^{-1}]]$ which acts as the identity on all components except $a$-th and as the matrix $T(u)$ on the $a$-th one.

Let the Yang $R$-matrix be
\begin{align*}
    R(u) = 1-\mathcal{P}u^{-1},
\end{align*}where $\mathcal{P}$ is the natural operator of transposition on $\mathbb{C}^{n} \otimes \mathbb{C}^{n}$.

\begin{theorem}[\textrm{RTT}-relation] \cite[Proposition 1.2.2]{Molev}
    Relations \eqref{t_ijk_relations}, which define the Yangian, are equivalent to the next relation in $(\End{n})^{\otimes (2)}\otimes\yang{n}[[u^{- 1},v^{-1}, (u-v)^{-1}]]$
    \begin{align*}
        R(u-v)T_1(u)T_2(v) = T_2(v)T_1(u)R(u-v).
    \end{align*}
\end{theorem}

For any matrix $B(u)$ with elements from $\yang{n}[[u^{-1}]]$ by the notation
\begin{align*}
    T(u) \to B(u)
\end{align*}we mean a function $\yang{n}\to\yang{n}$ mapping $t_{ij}(u)$ to $ij$-th entry of the matrix $B(u)$. 

\begin{proposition}
    Let $f(u)\in\mathbb{C}[[u^{-1}]]$ be an arbitrary invertible formal series and $c\in\mathbb{C}$ be an arbitrary complex number, then maps
    \begin{align}
        T(u) \to f(u)T(u),\label{multiplication_automorphism}\\
        T(u) \to T(u-c)\label{shift_automorphism}
    \end{align} are automorphisms of $\yang{n}$.
\end{proposition}

The structure of a Hopf algebra on the Yangian is well defined. We need only a formula of comultiplication:

\begin{align}
\label{comultiplication}
    \Delta(t_{ij}(u)) = \sum_{k=1}^{n} t_{ik}(u) \otimes t_{kj}(u),
\end{align} see \cite[1 ch., 5 par.]{Molev} for more details.

\subsection{Center of \texorpdfstring{$\yang{n}$}{Y(gl n)}} We refer the reader to \cite[1 ch., 6-7 par.]{Molev} for details and proofs.

\begin{definition}
\label{quantum_determinant}
    The quantum determinant of the matrix $T(u)$ is a formal series defined by one of the equivalent formulas
    \begin{multline*}
        \qdet{u} = \sgn{q}\sum_{\rho\in S_n}\sgn{\rho} t_{\rho(1)q(1)}(u) \cdot \ldots \cdot t_{\rho(n)q(n)}(u-n+1) =\\= \sgn{q}\sum_{\rho\in S_n}\sgn{\rho} t_{q(1)\rho(1)}(u-n+1) \cdot\ldots\cdot t_{q(n)\rho(n)}(u)
    \end{multline*} with $q\in S_n$ being some fixed in advance permutation.
\end{definition}

\begin{definition}
    A quantum minor of the matrix $T(u)$ corresponding to a set of rows $\{a_1,\ldots,a_m\}$ and a set of columns $\{b_1,\ldots,b_m\}$ with $m\leq n$ and $a_i,b_j\in\{1,\ldots,n\}$ is a formal series given by one of the equivalent formulas
    \begin{multline*}
        t^{a_1\ldots a_m}_{b_1\ldots b_m}(u) = \sum_{\rho\in S_m}\sgn{\rho} t_{a_{\rho(1)}b_1}(u) \cdot \ldots \cdot t_{a_{\rho(m)}b_m}(u-m+1) =\\= \sum_{\rho\in S_m}\sgn{\rho} t_{a_1b_{\rho(1)}}(u-m+1) \cdot\ldots\cdot t_{a_mb_{\rho(m)}}(u).
    \end{multline*}
\end{definition}

\begin{theorem}[{\cite[Th. 1.7.5]{Molev}}]
\label{theorem_of_center_yangian}
    The coefficients of the quantum determinant $\qdet{u}$ are algebraically independent and generate the center of the Yangian $\yang{n}$.
\end{theorem}

\begin{definition}
\label{gelfand_tsetlin_subalgebra_of_yangian}
    The Gelfand-Tsetlin subalgebra of $\yang{n}$ is a subalgebra generated by centers of $\yang{k},~ 1 \leq k \leq n$ corresponding to the chain of embeddings \eqref{yangian_embeddings}.
\end{definition} We will also denote that subalgebra by $\gz{n}$ when no confusion can arise. By Theorem \ref{theorem_of_center_yangian}, the Gelfand-Tsetlin subalgebra of $\yang{n}$ is generated by coefficients of the quantum minors
\begin{align*}
    t_{1\ldots i}^{1\ldots i}(u) ~\text{for all}~i\in\{1,\ldots,n\}.
\end{align*} 

\subsection{Drinfeld new realization.}There is another representation of $\yang{n}$ by generators and relations. 
Consider the following elements of $\yang{n}[[u^{-1}]]$:
\begin{align}
    h_i(u) = t^{1\ldots i}_{1\ldots i}(u+i-1)\cdot\big(t^{1\ldots i-1}_{1\ldots i-1}(u+i-1)\big)^{-1},\\
    f_i(u) = t^{1\ldots i-1,i+1}_{1\ldots i}(u+i-1)\cdot\big(t^{1\ldots i}_{1\ldots i}(u+i-1)\big)^{-1},\\
    e_i(u)=\big(t^{1\ldots i}_{1\ldots i}(u+i-1)\big)^{-1}\cdot t_{1\ldots i-1,i+1}^{1\ldots i}(u+i-1),
\end{align} where $h_i(u)$ is defined for $1\leq i \leq n$ and both $f_i(u)$ and $e_i(u)$ are defined for $1\leq i \leq n-1$.

We denote the $u^{-n}$ term of $h_i(u)$, $f_i(u)$ and $e_i(u)$ by $h_i^{(n)}$, $f_i^{(n)}$ and $e_i^{(n)}$ respectively. Notice that the formal series $h_i(u)$ starts with a constant term $1$, thus, it is invertible. For convenience, we denote $h_i^{-1}(u)$ by $h_i^{\prime}(u)$ and its coefficients by $h_i^{\prime\,(n)}$. 

\begin{theorem}[{\cite[Th. 3.1.5]{Molev}}]
   \label{new_realization}
   The coefficients of the formal series $h_i(u)$, $h_i^{\prime}(u)$ with $1\leq i \leq n$, and $f_i(u)$, $e_i(u)$ with $1\leq i \leq n-1$ are generators of $\yang{n}$ subject to the relations
   \begin{gather}
       h_i^{(0)}=1, ~~h_i(u)\cdot h_i^{\prime}(u)=1,~~ [h_i(u), h_j(u)]=0,\\
       [h_i(u),e_j(v)]=0, ~~ [h_i(u), f_j(v)]=0 ~~\text{if}~~i\not=j, j+1,\label{drinfeld_commutations}\\
       [e_i(u), e_j(v)]=0, ~~ [f_i(u), f_j(v)]=0, ~~\text{if}~~|i-j|>1,\\
       [e_i(u),f_j(v)]=0 ~~\text{if}~~i\not=j,\\
       (u-v)[e_i(u), f_i(v)] = h_i^{\prime}(u)h_{i+1}(u) - h_i^{\prime}(v)h_{i+1}(v), \\
       (u-v)[h_i(u), e_i(v)] = h_i(u)\big(e_i(v)-e_i(u)\big),\\
       (u-v)[h_i(u), f_i(v)] = \big(f_i(v)-f_i(u)\big) h_i(u),\\
       (u-v)[e_i(u),e_i(v)]=\big(e_i(v)-e_i(u)\big)^2,\\
       (u-v)[f_i(u),f_i(v)]=-\big(f_i(v)-f_i(u)\big)^2,\\
       [[e_i(u),e_j(v)], e_j(w)] + [[e_i(u),e_j(w)], e_j(v)]=0 ~~\text{if}~~|i-j|=1,\\
       [e_i(u),[e_i(v), e_j(w)]] + [e_i(v),[e_i(u), e_j(w)]] = 0 ~~\text{if}~~|i-j|=1,\\
       [[f_i(u),f_j(v)], f_j(w)] + [[f_i(u),f_j(w)], f_j(v)]=0 ~~\text{if}~~|i-j|=1,\\
       [f_i(u),[f_i(v), f_j(w)]] + [f_i(v),[f_i(u), f_j(w)]] = 0 ~~\text{if}~~|i-j|=1,\\
       [e_i^{(r)}, e_{i+1}^{(s+1)}] - [e_i^{(r+1)}, e_{i+1}^{(s)}] = -e_i^{(r)}e_{i+1}^{(s)},\\
       [f_i^{(r+1)}, f_{i+1}^{(s)}] - [f_i^{(r)}, f_{i+1}^{(s+1)}] = -f_{i+1}^{(s)}f_{i}^{(r)}
   \end{gather} for all possible indices $1\leq i,\,j \leq n$.
\end{theorem}

It is straightforward that $\gz{n}$ is generated by the coefficients of the formal series $h_i(u)$. Since the evaluation map is surjective, the image of that subalgebra is the Gelfand-Tsetlin subalgebra $\gz{n}$ of $\uni{n}$, defined in \ref{gelfand_tsetlin_subalgebra}.

An easy computation shows that under the evaluation homomorphism we have 

\begin{gather}
    \pi\big(h_i(u)\big) = \pi\big(1+\t{i}{i}{1}u^{-1}+\ldots\big) = 1 + E_{ii}u^{-1}+\ldots,\label{row_of_h_i}\\
    \pi\big(f_i(u)\big) = \pi\big(\t{i+1}{i}{1}u^{-1}+\ldots\big) =E_{i+1,i}u^{-1}+\ldots,\label{row_of_f_i}\\
    \pi\big(h_i(u)\big) = \pi\big(\t{i}{i+1}{1}u^{-1}+\ldots\big) = E_{i,i+1}u^{-1}+\ldots.\label{row_of_e_i}
\end{gather}
The elements $E_{i+1,i}$, $E_{i,i+1}$ and $E_{ii}$ are the standard generators of $\gl{n}$, thus the Drinfeld generators will be used to deduce Gelfand-Tsetlin formulas.

For a fixed index $1\leq i \leq n-1$ consider a subalgebra $\yang{2}^{(i)}$ of $\yang{n}$ generated by coefficients of the formal series
\begin{align*}
    e_i(u), f_i(u), h_i(u), h_{i+1}(u).
\end{align*}

Notice that the relations of $\yang{2}$ from Theorem \ref{new_realization} hold under the replacement of indices $(1,2)$ with $(i, i+1)$ for $1\leq i\leq n-1$. Therefore, we have the next proposition.

\begin{proposition}
\label{definition_of_yang2i}
    The subalgebra $\yang{2}^{(i)}$ is isomorphic to $\yang{2}$ under the map
    \begin{gather*}
        h_1(u) \mapsto h_i(u), ~h_2(u) \mapsto h_{i+1}(u), \\
        f_1(u) \mapsto f_i(u), ~ e_1(u) \mapsto e_i(u).
    \end{gather*}
\end{proposition}

\subsection{Representations of \texorpdfstring{$\yang{n}$}{Y(gl n)}}
Let $a(u) = \sum_{i=0}^{+\infty}a_iu^{-i}$ be an arbitrary formal series of $\yang{n}[[u^{-1}]]$, $V$ some $\yang{n}$-module. We define for each $v\in V$
\begin{align*}
    a(u)\cdot v = \suminf{j}{0} (a_j \cdot v) u^{-j} \in V[[u^{-1}]].
\end{align*}

\begin{definition}
    A $\yang{n}$-module $V$ is called a representation of the highest weight $\lambda$ with
    \begin{align*}
        \lambda = \big(\lambda_1(u), \ldots, \lambda_n(u)\big), ~\text{where}~\lambda_i(u)\in\mathbb{C}[[u^{-1}]],
    \end{align*} if one can find a vector $\zeta$ such that $\yang{n}\cdot \zeta$ equals $V$ and
    \begin{align*}
        h_i(u)\cdot \zeta=\lambda_i(u)\,\zeta, ~~~\forall~ i\in\{1,\ldots,n\},\\
        e_i(u)\cdot \zeta=0, ~~~\forall~ i\in\{1,\ldots,n-1\}.
    \end{align*} In this way we shall call $\zeta$ the highest vector. It is unique up to scalar multiplication.
\end{definition}

The description of representation of the highest weight can be given in the \textrm{RTT}-relation as well.

\begin{proposition}[{\cite[Pr. 3.2.2]{Molev}}]
\label{action_of_tij}
    The $\yang{n}$-module $V$ is a representation of the highest weight $\lambda$ if and only if one can find a vector $\widetilde{\zeta}$ such that $\yang{n}\cdot \widetilde{\zeta}$ equals to $V$ and
    \begin{align*}
        t_{ii}(u)\cdot \widetilde{\zeta}=\lambda_i(u)\zeta ~~~\forall 1\leq i\leq n,\\
        t_{ij}(u)\cdot \widetilde{\zeta}=0 ~~~\forall~ 1\leq i\leq j\leq n.
    \end{align*}
    Moreover, $\tilde \zeta = c \zeta, c \in \mathbb{C}\setminus\{0\}$.
\end{proposition}

The following theorem is an analogue of Theorem \ref{gl_n_classification_theorem} for $\gl{n}$-modules.
\begin{theorem}[{\cite[Th. 3.2.7]{Molev}}]
\label{theorem_of_highest_weight_repr}
    \begin{enumerate}
        \item Every finite-dimensional irreducible representation of $\yang{n}$ is a representation of the highest weight.
        \item For each $\lambda(u) \in 1+ u^{-1}\mathbb{C}[[u^{-1}]]$ there is a unique irreducible representation $V_{\lambda(u)}$ of the highest weight $\lambda(u)$.
    \end{enumerate}
\end{theorem}

The next proposition yields information about an action of the Yangian center.
\begin{proposition}
\label{action_of_quantum_determinant}
    Let $V_{\lambda}$ be a finite-dimensional irreducible representation of the highest weight $\lambda$. Then the coefficients of the quantum determinant act on $V_{\lambda}$ by the next rule:
    \begin{align*}
        \qdet{u} \big|_{V_{\lambda}} = \lambda_1(u)\cdot\ldots\cdot\lambda_n(u-n+1).
    \end{align*}
\end{proposition}

\begin{proof}
    By Schur's lemma, each coefficient of the Yangian center acts on $V_{\lambda}$ by multiplying on some scalar. Thus, it suffices to compute the action on the highest vector $\zeta$:
    \begin{align*}
        \sum_{\rho\in S_n}\sgn{\rho} t_{\rho(1)1}(u) \cdot \ldots \cdot t_{\rho(n)n}(u-n+1) \cdot \zeta.
    \end{align*} On account of Proposition \ref{action_of_tij}, we have
    \begin{align*}
        t_{\rho(n)n}(u-n+1)\cdot\zeta = \delta_{\rho(n)n}\lambda_n(u-n+1)\zeta,
    \end{align*} because $\rho(n)$ is less than or equal to $n$ for all $\rho\in S_n$. Identifying $S_{n-1}$ with a subgroup of $S_{n}$ which does not act on $n$ we deduce that
    \begin{align*}
        \qdet{u} \cdot \zeta = \lambda_n(u-n+1)\sum_{\rho\in S_{n-1}}\sgn{\rho}t_{\rho(1)1}(u) \cdot \ldots \cdot t_{\rho(n-1)n-1}(u-n+2) \cdot \zeta.
    \end{align*} The same argument remains valid for each $i\leq n$ and the proof is complete.
\end{proof}

Now let us introduce a class of irreducible $\yang{n}$-modules constructed by the evaluation homomorphism which was defined in \ref{evaluation_homomorphism}.

\begin{definition}
\label{evaluation_representation}
    Suppose that $V_{\lambda}$ is an irreducible representation of $\gl{n}$ of the highest weight
    \begin{align}
    \label{highest_weight_gln}
        \lambda = \big(\lambda_1,\ldots,\lambda_n\big).
    \end{align} Then $V_{\lambda}$ can be considered as an irreducible $\yang{n}$-module by using the evaluation homomorphism $\pi$. It is a simple matter to check that the $\gl{n}$-highest vector is the $\yang{n}$-highest vector and that the highest weight of $V_{\lambda}$ as $\yang{n}$-module is equal to
    \begin{align}
    \label{highest_weight_yang}
        \big(1+u^{-1}\lambda_1, \ldots, 1+u^{-1}\lambda_n\big).
    \end{align}

    It will cause no confusion if we use a notation $V_{\lambda}$ as the irreducible $\gl{n}$-module and the $\yang{n}$-module and $\lambda$ as \eqref{highest_weight_gln} and \eqref{highest_weight_yang} depending on context.
\end{definition}

We discussed in \hyperref[Section2]{Section 2} that any finite-dimensional irreducible $\gl{n}$-module $V_{\lambda}$ decomposes into a direct sum of irreducible $\gl{n-1}$-modules
\begin{align*}
    V_{\lambda}\big|_{\gl{n-1}} = \bigoplus_{\mu} V_{\mu}.
\end{align*} Consider $V_{\lambda}$ as the $\yang{n}$-module. Since the chains of embedding $\uni{n}$ and $\yang{n}$ are agreed with the evaluation homomorphism, we conclude that
\begin{align*}
    V_{\lambda}\big|_{\yang{n-1}} = \bigoplus_{\mu} V_{\mu},
\end{align*} where $V_{\mu}$ are corresponding irreducible evaluation representations of $\yang{n-1}$.

We defined the comultiplication $\Delta$ of $\yang{n}$ in \ref{comultiplication}, thus we are able to consider a tensor product of $\yang{n}$-module. We shall observe that $\Delta$ is not co-commutative, thus tensor product is not commutative in general. 

\begin{proposition}
\label{independence_of_order}
    Let $V_{\lambda^{(1)}}, \ldots, V_{\lambda^{(k)}}$ be the finite-dimensional evaluation irreducible $\yang{n}$-modules and consider a module
    \begin{align*}
        V = V_{\lambda^{(1)}} \otimes \ldots \otimes V_{\lambda^{(k)}}.
    \end{align*} Then the generators $\t{i}{j}{r}$ act by zero unless $r \leq k$ and if $V$ is irreducible, then the reordering the tensor factors gives a representation isomorphic to $V$.
\end{proposition}

\newcommand{\Lira}[1]{\Lir{\alpha_{#1}}{\beta_{#1}}}
\subsection{Classification of \texorpdfstring{$\yang{2}$}{Y(gl 2)}-representations}
We follow the notation of \cite{Molev} and use $\Lir{\alpha}{\beta}$ to denote the evaluation representation of $\yang{2}$ of the highest weight $(\alpha,\beta)$.

\begin{definition}
\label{string_definition}
    The string $\mathit{S}(\alpha,\beta)$ of two complex numbers $\alpha$ and $\beta$ such that
    \begin{align*}
        \alpha-\beta\in\mathbb{Z}_{\geq 0}
    \end{align*} is a set $\{\beta, \beta+1,\ldots, \alpha-1\} \subset \mathbb{C}$.
\end{definition}

\begin{definition}
    We call two strings $\mathcal{S}_{1}$ and $\mathcal{S}_{2}$ to be in general position if one of next conditions is true
    \begin{itemize}
        \item $\mathcal{S}_{1}\subset \mathcal{S}_{2}$ or $\mathcal{S}_{2}\subset\mathcal{S}_{1}$,
        \item a set $\mathcal{S}_{1}\cup\mathcal{S}_{2}$ is not a string for any numbers.
    \end{itemize}
\end{definition}

\begin{theorem}[{\cite[Cor. 3.3.6]{Molev}}]
\label{theorem_string_conditions}
    The $\yang{2}$-module
    \begin{align*}
        \Lira{1} \otimes \ldots \otimes \Lira{k}
    \end{align*} is irreducible if and only if each pair of the corresponding strings $\mathcal{S}(\alpha_i,\beta_i)$ is in general position.
\end{theorem}

The next theorem provides a full classification of finite-dimensional irreducible representations of $\yang{2}$.

\begin{theorem}[{\cite[Th. 3.3.3]{Molev}}]
\label{classification_theorem_yang2}
    An irreducible representation $V$ of the highest weight $\big(\lambda_1(u), \lambda_2(u)\big)$ is finite-dimensional if and only if one can find a polynomial $P(u)\in\mathbb{C}[u]$ such that
    \begin{align*}
        \frac{\lambda_1(u)}{\lambda_2(u)} = \frac{P(u+1)}{P(u)}.
    \end{align*} In that case, there exists a formal power series $f(u)\in\mathbb{C}[[u^{-1}]]$ with constant term equal to $1$ such that after composition with the automorphism of multiplication \ref{multiplication_automorphism} the representation is isomorphic to
    \begin{align*}
        \Lira{1} \otimes \ldots \otimes \Lira{k}
    \end{align*} for some pairs $\big(\alpha_i, \beta_i\big)\in\mathbb{C}^{2}$. The pair $\big(f(u), P(u)\big)$ uniquely corresponds with the representation.

    Conversely, a pair $\big(f(u), P(u)\big)$ with arbitrary $f(u)\in 1+u^{-1}\cdot\mathbb{C}[[u^{-1}]]$ and $P(u)\in\mathbb{C}[u]$ with the leading term equal to $1$ corresponds to a certain finite-dimensional irreducible representation.
\end{theorem}

The polynomial $P(u)$ is called the Drinfeld polynomial. The condition for $f(u)$ can be written in a way
\begin{align*}
        f(u)\lambda_1(u) = (1+\alpha_1u^{-1})\cdot\ldots\cdot(1+\alpha_ku^{-1}),\\
        f(u)\lambda_2(u) = (1+\beta_1u^{-1})\cdot\ldots\cdot(1+\beta_ku^{-1}).
\end{align*} The quotient $\frac{\lambda_{1}(u)}{\lambda_{2}(u)}$ does not change after the automorphism of multiplication, thus, the Drinfeld polynomial corresponds to a tensor product
\begin{align*}
    \Lira{1} \otimes \ldots \otimes \Lira{k}.
\end{align*} Namely the roots of $P(-u)$ is exactly the elements of the corresponding strings $\mathcal{S}(\alpha_i, \beta_i)$.

%% file: 4_GZ_Basis_for_Y_gl2.tex
\subsection{Explicit realization of certain irreducible representations \texorpdfstring{$\yang{2}$}{Y(gl 2)}}

Let us consider an irreducible representation of the type
\begin{align*}
    V=\Lira{1}\otimes\ldots\otimes\Lira{k},
\end{align*} which corresponds to the pair $\big(1, P(u)\big)$: for an arbitrary representation the received formulas for $t_{ij}(u)$ should be multiplied by $f(u)$, see Theorem \ref{classification_theorem_yang2} and discussion further.

We strengthen the assumption on the corresponding strings given in Theorem \ref{theorem_string_conditions}; from now on, it is required that each pair of the strings $\mathcal{S}(\alpha_i, \beta_i)$ does not intersect.

We show that the Gelfand-Tsetlin subalgebra of $\yang{2}$ acts by semisimple operators in such representations. Moreover, that assumption on intersections is equivalent to semi-simplicity of $\gz{2}$ action, see \cite[Th. 4.1]{NazarovTarasov2}.

We prove the following auxiliary proposition.

\begin{proposition}
    \label{gz_eigen_basis_proposition}
    Suppose that $W$ is a finite-dimensional irreducible $\yang{2}$-module of the highest weight.
    \begin{enumerate}
        \item The basis of representation $W$ is eigenbasis for $\gz{2}$ if and only if it is eigenbasis for $t_{11}(u)$;
        \item the action of $t_{11}(u)$, $t_{12}(u)$ and $t_{21}(u)$ in that eigenbasis is effectively describe the action of $\yang{2}$.
    \end{enumerate}
\end{proposition}

\begin{proof}
    Both statements are a consequence of Proposition \ref{action_of_quantum_determinant}, which asserts that $\qdet{u}$ acts by scalar in $W$. Namely the subalgebra $\gz{2}$ generates by $t_{11}(u)$ and $\qdet{u}$, so the first statement is proved.

    For the second statement it is suffices to prove that we can compute the action of $t_{22}(u)$, and it is easy to check that
    \begin{align}
        \label{formula_for_t_22_by_another}
        t_{22}(u) = \bigg(\qdet{u} \cdot \big(t_{11}(u)\big)^{-1} + t_{21}(u)\cdot \big(t_{11}(u)\big)^{-1} \cdot t_{12}(u)\bigg).
    \end{align} Since the each basis vector $v_{r}$ are multiplied by some $\mu_r{(u)}\in\mathbb{C}[[u^{-1}]]$ under the action of $t_{11}(u)$, the action of $\big(t_{11}(u)\big)^{-1}$ is equal to the multiplication by $\mu_r{(u)}^{-1}$, and the action of $t_{22}(u)$ can be easily computed by formula \eqref{formula_for_t_22_by_another}. 
\end{proof}

\subsection{Construction of basis}
We recall that according to Proposition \ref{independence_of_order} elements $\t{i}{j}{r}$ act on $V$ by zero unless $r\leq k$. For this reason, let us introduce the new formal series
\begin{align*}
    T_{ij}(u) = u^{k}t_{ij}(u),
\end{align*} which acts on $V$ by some polynomial of degree less than or equal $k$ if $i=j$ and strictly less than $k$ otherwise. 

Obviously, the same relations as for $t_{ij}(u)$ hold for the new series:
\begin{align}
\label{T_iju_relations}
    (u-v)[T_{ij}(u), T_{kl}(v)] = T_{kj}(u)T_{il}(v) - T_{kj}(v)T_{il}(u).
\end{align}

Assume that $\zeta$ is the highest vector of $V$. For any set of indices
\begin{align*}
    r=(r_1,\ldots,r_k)
\end{align*} with $r_i\in\{0,\ldots,\alpha_i-\beta_i\}$ (such indices are said to be allowable) we define the next vectors from $V$:
\begin{align}
\label{v_r_definition}
    v_r = \prod_{i=1}^{k} \big(T_{21}(-\alpha_i+r_i-1) \cdot \ldots\cdot T_{21}(-\alpha_i)\big) \cdot \zeta,
\end{align} where the corresponding product is empty if $r_i$ equals $0$. We observe that $T_{21}(u)$ and $T_{21}(v)$ commute under relations \eqref{T_iju_relations} thus the order of the factors in \eqref{v_r_definition} does not matter.

The next construction is adapted from \cite[Th. 3.3.8]{Molev}.

\begin{theorem}
\label{yangian_2_GF_realization}
    Suppose that 
    \begin{align*}
        V = \Lira{1}\otimes\ldots\otimes \Lira{k}
    \end{align*} is a finite-dimensional irreducible representation and each pair of the corresponding strings $\mathit{S}(\alpha_i, \beta_i)$ does not intersect. Then the vectors $v_r$, defined in \eqref{v_r_definition}, for all allowable indices are a basis of $V$ and formulas of action are
    \begin{gather}
        T_{11}(u)\cdot v_r = (u+\alpha_1-r_1) \cdot\ldots\cdot (u+\alpha_k-r_k) v_r,\label{T_11_formula}\\
        T_{21}(-\alpha_i+r_i)\cdot v_r = v_{r+\delta_i},\label{T_21_formula}\\
        T_{12}(-\alpha_i+r_i) \cdot v_r= -\prod_{j=1}^{k}(\alpha_j-\alpha_i+r_i)(\beta_j-\alpha_i+r_i-1) \cdot v_{r-\delta_i},\label{T_12_formula}
    \end{gather} where $r\pm\delta_i$ differs from $r$ only in the $i$-th place by adding $\pm1$ and if $r$ is not a set of allowable indices, then the corresponding $v_r$ equals $0$.
\end{theorem}

\begin{proof}
    If $r+\delta_i$ is a set of allowable indices, then formula \eqref{T_21_formula} for $T_{21}(u)$ follows from Definition \ref{v_r_definition} of the vectors $v_r$.

    The proof of formula \eqref{T_11_formula} for $T_{11}(u)$ is by induction on sum $r_1+\ldots+r_k$. The base case is
    \begin{align*}
        r_1+\ldots+r_k=0,
    \end{align*} hence each of $r_i$ equals to zero. The vector $v_{0,\ldots,0}$ is the highest vector $\zeta$ of $V$ by definition. Then formula \eqref{T_11_formula}
    \begin{align*}
        T_{11}(u) \cdot \zeta = u^k t_{11}(u)\cdot\zeta = (u+\alpha_1)\cdot\ldots\cdot (u+\alpha_k)\zeta
    \end{align*} is a consequence of Proposition \ref{action_of_tij}.

    In the induction step, we can certainly assume that $r_j\not=0$, then by relation \eqref{T_iju_relations} and the definition of $v_r$
    \begin{multline*}
        T_{11}(u)\cdot v_{r} = T_{11}(u) T_{21}(-\alpha_j+r_j-1) \cdot v_{r-\delta_j} = \\
        = T_{21}(-\alpha_j+r_j-1)T_{11}(u) \cdot v_{r-\delta_j} + [T_{11}(u), T_{21}(-\alpha_j+r_j-1)]\cdot v_{r-\delta_j} =\\
        = (u+\alpha_1-r_1) \cdot \ldots \cdot (u+\alpha_j-r_j+1) \cdot \ldots \cdot (u+\alpha_k-r_k) v_r + \\
        +\frac{1}{u+\alpha_j-r_j+1}\big(T_{21}(u)T_{11}(-\alpha_j+r_j-1) v_{r-\delta_j} - T_{21}(-\alpha_j+r_j-1)T_{11}(u)\cdot v_{r-\delta_j}\big)=\\
        = (u+\alpha_1-r_1) \cdot \ldots \cdot (u+\alpha_j-r_j+1) \cdot \ldots \cdot (u+\alpha_k-r_k) v_r -\\
        - (u+\alpha_1-r_1) \cdot \ldots \cdot \land^{j} \cdot \ldots \cdot (u+\alpha_k-r_k) v_r =\\
        = (u+\alpha_1-r_1) \cdot\ldots\cdot (u+\alpha_k-r_k)v_r,
    \end{multline*} where $\land^{j}$ means a skip the $j$-th factor, $T_{11}(-\alpha_j+r_j-1)$ is equal to zero by the induction conclusion, and \eqref{T_11_formula} is proved.

    The proof of formula \eqref{T_12_formula} falls into two steps: $r_i=0$ and $r_i\not=0$. Suppose that $r_i=0$, then the assertion of the formula is
    \begin{align*}
        T_{12}(-\alpha_i)\cdot v_{r}=0.
    \end{align*} The argument is by induction on sum $r_1+\ldots+r_k$ as well. According to the above remark, the base case follows from Proposition \ref{action_of_tij} because $v_{0,\ldots,0}$ is the highest vector:
    \begin{align*}
        T_{12}(u)\cdot \zeta = u^kt_{12}(u)\cdot\zeta=0.
    \end{align*} In the induction step, we can certainly assume that $r_j\not=0$, then by relation \eqref{T_iju_relations} and the definition of $v_r$
    \begin{multline*}
        T_{12}(-\alpha_i)\cdot v_r = T_{12}(-\alpha_i) T_{21}(-\alpha_j+r_j-1)\cdot v_{r-\delta_j} = \\
        = T_{21}(-\alpha_j+r_j-1)T_{12}(-\alpha_i)\cdot v_{r-\delta_j} + [T_{12}(-\alpha_i), T_{21}(-\alpha_j+r_j-1)] =\\
        = \frac{1}{\alpha_j-\alpha_i-r_j+1}\big(T_{22}(-\alpha_i)T_{11}(-\alpha_j+r_j-1)v_{r-\delta_j} -\\- T_{22}(-\alpha_j+r_j-1)T_{11}(-\alpha_i)\cdot v_{r-\delta_j}\big) = 0.
    \end{multline*} where $T_{12}(-\alpha_i)\cdot v_{r-\delta_j}$ is equal to zero by the induction conclusion, and both $T_{11}(...)$ equal zero by already proved formula \eqref{T_11_formula}. We shall emphasize that the denominator $(\alpha_j-\alpha_i-r_j+1)$ is not equal to zero since the corresponding strings do not intersect.

    We now turn to the case $r_i\not=0$. Let us recall the quantum determinant in a case of $\yang{2}$, according to one variant of Definition \ref{quantum_determinant}:
    \begin{align*}
        u^k(u-1)^k \qdet{u} = T_{22}(u)T_{11}(u-1)-T_{12}(u)T_{21}(u-1).
    \end{align*} We know the formula of the quantum determinant action from Proposition \ref{action_of_quantum_determinant}:
    \begin{align*}
        u^k(u-1)^k \qdet{u}\big|_{V} = \prod_{j=1}^{k} (u+\alpha_j)(u+\beta_j-1).
    \end{align*} Now applying that formula in a case $u=-\alpha_i+r_i$ for the vector $v_{r-\delta_i}$ and using already proved formula \eqref{T_11_formula} for $T_{11}(u)$ we deduce that
    \begin{multline*}
        u^k(u-1)^k \qdet{u}\big|_{u=-\alpha_i+r_i} \cdot v_{r-\delta_i} =\\= -T_{12}(-\alpha_i+r_i)T_{21}(-\alpha_i+r_i-1)\cdot v_{r-\delta_i}
        = -T_{12}(-\alpha_i+r_i)\cdot v_r,
    \end{multline*} what establishes formula \eqref{T_12_formula} for $T_{12}(u)$
    \begin{align*}
        T_{12}(-\alpha_i+r_i) \cdot v_r= -\prod_{j=1}^{k}(\alpha_j-\alpha_i+r_i)(\beta_j-\alpha_i+r_i-1) \cdot v_{r-\delta_i}.
    \end{align*}

    It remains to prove that the set of the vectors $\{v_r\}$ with allowable indices is a basis of $V$. It is easy to see that by acting $T_{12}(u)$ several times with the right $u$ we can make the highest vector $\zeta$ of any $v_r$ with non-zero scalar since each pair of the corresponding strings does not intersect. Thus the vectors $\{v_r\}$ are all non-zero.

    In that case, the vectors $\{v_r\}$ are linearly independent because they are eigenvectors of $T_{11}(u)$ with different eigenvalues. Finally, an amount of $\{v_r\}$ which is equal to an amount of allowable indices is
    \begin{align*}
        \prod_{j=1}^{k}(\alpha_j-\beta_j+1),
    \end{align*} which is the dimension of $V$. Therefore, $\{v_r\}$ is a basis, formula \eqref{T_21_formula} for $T_{21}(u)$ holds in a case of unallowable indices, and the proof is complete.
\end{proof}

On account of Proposition \ref{gz_eigen_basis_proposition}, we immediately conclude the next corollaries.
\begin{corollary}
    The basis $v_r$ is an eigenbasis for $\gz{2}$.
\end{corollary}

\begin{corollary}
\label{corollary_that_determine_uniquelt_yang2}
    Formulas \eqref{T_11_formula}, \eqref{T_21_formula} and \eqref{T_12_formula} proved above effectively determine the action of $\yang{2}$.
\end{corollary}

\begin{proof}
    As a consequence of Proposition \ref{independence_of_order}, the formal series $T_{21}(u)$ and $T_{12}(u)$ act in $V$ by polynomials of degree less than $k$. Fixing some basis vector $v_r$, we know the action both of $T_{21}(u)$ and $T_{12}(u)$ in $k$ points $-\alpha_i+r_i$, which are all different since each pair of the corresponding strings does not intersect.

    We thus can write a Lagrange interpolating polynomial for defining the action of $T_{21}(u)$ and $T_{12}(u)$:
    \begin{gather}
    \label{T_21_whole_formula_with_u}
        T_{21}(u) \cdot v_r = \sum_{j=1}^{k}\frac{(u+\alpha_1-r_1)\cdot\ldots\land^{j}\ldots\cdot(u+\alpha_k-r_k)}{(\alpha_1-\alpha_j+r_j-r_1)\cdot\ldots\land^{j}\ldots\cdot(\alpha_k-\alpha_j+r_j-r_k)}v_{r+\delta_j},
    \end{gather}
    \begin{multline}
    \label{T_12_whole_formula_with_u}
        T_{12}(u) \cdot v_r = -\sum_{j=1}^{k}\prod_{m=1}^{k}(\alpha_m-\alpha_j+r_j)(\beta_m-\alpha_j+r_j-1) \cdot \\
        \cdot \frac{(u+\alpha_1-r_1)\cdot\ldots\land^{j}\ldots\cdot(u+\alpha_k-r_k)}{(\alpha_1-\alpha_j+r_j-r_1)\cdot\ldots\land^{j}\ldots\cdot(\alpha_k-\alpha_j+r_j-r_k)}v_{r-\delta_j},
    \end{multline} with the notation $\land^{j}$ means a skip the $j$-th factor.

    As a consequence of Proposition \ref{gz_eigen_basis_proposition}, the action of $\yang{2}$ is effectively described.
\end{proof}

\begin{corollary}
\label{sl_n_drinfeld_formulas_yang2}
    The formulas in the Drinfeld realization are
    \begin{align}
    \label{f_1_whole_formula_with_u}
        f_1(u) \cdot v_r = \sum_{j=1}^{k}\frac{(u+\alpha_j+r_j-1)^{-1}}{(\alpha_1-\alpha_j+r_j-r_1)\cdot\ldots\land^{j}\ldots\cdot(\alpha_k-\alpha_j+r_j-r_k)}v_{r+\delta_j},
    \end{align}
    \begin{multline}
    \label{e_1_whole_formula_with_u}
        e_1(u) \cdot v_r = -\sum_{j=1}^{k}\prod_{m=1}^{k}(\alpha_m-\alpha_j+r_j)(\beta_m-\alpha_j+r_j-1) \cdot \\
        \cdot \frac{(u+\alpha_j-r_j+1)^{-1}}{(\alpha_1-\alpha_j+r_j-r_1)\cdot\ldots\land^{j}\ldots\cdot(\alpha_k-\alpha_j+r_j-r_k)}v_{r-\delta_j},
    \end{multline}with the same notation $\land^{j}$ means a skip the $j$-th factor. Moreover, these formulas are invariant under the multiplication automorphism. In other words, the formulas only depend on the Drinfeld polynomial $P(u)$ of the representation.
\end{corollary}

\begin{proof}
    By definition,
    \begin{align*}
        f_1(u)=t_{21}(u)\big(t_{11}(u)\big)^{-1},~e_1(u) = \big(t_{11}(u)\big)^{-1}t_{12}(u).
    \end{align*} Now the proof of the formulas is straightforward and it is evident that $f_1(u)$ and $e_1(u)$ are invariant under the automorphism of multiplication on any invertible formal series.
\end{proof}

\begin{remark}
\label{uniqeuly_defined_T11}
    The basis vector $v_r$ is uniquely defined up to scalar by the action of $T_{11}(u)$.
\end{remark}

\begin{proof}
    It is clear since each pair of the corresponding strings does not intersect. But we should emphasize that if
    \begin{align*}
        T_{11}(u) \cdot v_r = (u+a_1)\cdot\ldots\cdot(u+a_k)v_{r},
    \end{align*} then the index $r_j$ is defined by
    \begin{align*}
        \alpha_j - r_j = a_j
    \end{align*} maybe, after renumbering of the roots.
\end{proof}

%% file: 5_GZ_Basis_for_gln.tex
\subsection{The evaluation representation}
Let us fix a finite-dimensional irreducible representation $V_{\lambda}$ of $\gl{n}$ of the highest weight
\begin{align*}
    \lambda=\big(\lambda_1,\ldots,\lambda_n\big) \in \mathbb{Z}_{\geq 0}^n,
\end{align*} such that
\begin{align*}
    \lambda_1 \geq \ldots \geq \lambda_n \geq 0.
\end{align*}

As we discussed in \ref{evaluation_representation}, $V_{\lambda}$ can be endowed with a structure of irreducible $\yang{n}$-module by the evaluation homomorphism $\pi$, which is called the evaluation module, and it is agreed with $\uni{n}$-structure, in particular,
\begin{align*}
    V_{\lambda}\big|_{\yang{n-1}} = \bigoplus_{\mu} V_{\mu}
\end{align*}with $V_{\mu}$ being the evaluation $\yang{n-1}$-module with the highest weight $\mu$ such that
\begin{align*}
    \lambda_1 \geq \mu_1 \geq \lambda_2 \geq \ldots \lambda_{n-1} \geq \mu_{n-1} \geq \lambda_n \geq 0.
\end{align*}

We  consider $\yang{n}$-module $V_{\lambda}$ as a representation of subalgebras $\yang{2}^{(i)}$, see Proposition \ref{definition_of_yang2i}, because they contain the $\gl{n}$-generators $E_{i+1,i}$, $E_{i,i+1}$ and $E_{ii}$. We show that $V_{\lambda}$ can be decomposed into a direct sum of certain irreducible $\yang{2}^{(i)}$-modules, and applying formulas from \hyperref[Section4]{Section 4} we get the explicit realization. 

\subsection{Centralizer construction and complete reducibility}
Consider the following map:
\begin{align*}
    \psi: \yang{2} \to \yang{n}: t_{ij}(u) \to \left(t_{1\ldots n-2}^{1\ldots n-2}(u)\right)^{-1} \cdot t_{1\ldots n-2, n-2+j}^{1\ldots n-2, n-2+i}(u).
\end{align*} It turns out that $\psi$ is an injective algebra homomorphism, proof can be find in \cite[1 ch., 11 par.]{Molev}. Moreover, after composition with the evaluation homomorphism to $\uni{n}$ the image under $\pi\circ \psi$ is in the centralizer $\uni{n}^{\gl{n-2}}$ (that is a subalgebra consists of all elements of $\uni{n}$ commutes with the embedded $\uni{n-2}$).

We know that the $\gl{n}$-module $V_{\lambda}$ is completely reducible as $\gl{n-1}$-module without multiplicity (see \hyperref[Section2]{Section 2}). Applying that argument again, we get
\begin{align*}
    V_{\lambda}\big|_{\gl{n-2}} = \bigoplus_{\mu \subset \lambda} K_{\mu}\otimes V_{\mu},
\end{align*} where $V_{\mu}$ is an irreducible $\gl{n-2}$-module of the highest weight $\mu$ and $K_{\mu}$ is a multiplicity space which can be interpreted as 
\begin{align*}
    K_{\mu}\cong \mathrm{Hom}_{\gl{n-2}}(V_{\mu}, V_{\lambda}).
\end{align*} The space $K_{\mu}$ is endowed with a natural structure of $\uni{n}^{\gl{n-2}}$-module, and it is well-known fact that it is an irreducible representation of $\uni{n}^{\gl{n-2}}$ (see \cite[9 par.]{Dixmier}).

\begin{proposition}[{\cite[Pr. 8.5.1]{Molev}}]
    The multiplicity space $K_{\mu}$ can be endowed with a structure of an irreducible representation of $\yang{2}$ by using the homomorphism $\pi\circ \psi$.
\end{proposition}

\begin{proof}
    The proof can be found in \cite[9 ch., 3-5 par.]{Molev}. It is based on the observation that the center $Z\uni{n-2}=\pi({Z\yang{n-2}})$, which is obviously in $\uni{n}^{\gl{n-2}}$, acts by scalar in $K_{\mu}$.

    Therefore, if $\uni{n}^{\gl{n-2}}$ is generated by the image of $\pi\circ \psi(\yang{2})$ and the center $Z\uni{n-2}$, then $K_{\mu}$ is the irreducible representation of $\yang{2}$.
\end{proof}

One can notice that the elements of $\yang{2}^{(i)}$ seem to be shifted in comparison with the image of $\psi$. Thus, we need the next proposition.

\begin{proposition}
\label{shifted_subalgebra}
    Suppose that $\mathcal{K}$ is an associative $\mathbb{C}$-algebra and $a(u)\in\mathcal{K}[[u^{-1}]]$. Then the subalgebra of $\mathcal{K}$, generated by the coefficients of $a(u)$, is equal (as a set) to the subalgebra, generated by the coefficients of $a(u-c)$, where is $c\in\mathbb{C}$ is an arbitrary complex number.
\end{proposition}

\begin{proof}
    Let $a(u)=\sum\limits_{k=0}^{+\infty}a_ku^{-k}$. It is simple matter to compute that
    \begin{align*}
        (u-c)^{-k} = \frac{u^{-k}}{(1-cu^{-1})^{k}}=u^{-k}\sum_{n=0}^{+\infty}\binom{k+n-1}{k-1}c^nu^{-n}.
    \end{align*} Hence, the $n$-th term of $a(u-c)$ depends only on the first $n$-th summands and has the form
    \begin{align*}
        \mathcal{C}_{1}c^n\cdot a_1+\mathcal{C}_{2}c^{n-1}\cdot a_2 + \ldots + \mathcal{C}_{n-1}c\cdot a_{n-1} + a_n, ~\mathcal{C}_{i}\in\mathbb{Z}
    \end{align*}that is, the $n$-th term is the linear combination of the previous ones. So, the proof is complete.
\end{proof}

From now on, we can consider the shifted $\psi$:
\begin{align*}
    \psi: \yang{2} \to \yang{n}: t_{ij}(u) \to \left(t_{1\ldots n-2}^{1\ldots n-2}(u-n+2)\right)^{-1} \cdot t_{1\ldots n-2, n-2+j}^{1\ldots n-2, n-2+i}(u-n+2),
\end{align*} and by Proposition \ref{shifted_subalgebra}, all properties of the image of $\psi$ are the same.

Finally, we want to deduce that $K_{\mu}$ is an irreducible representation of $\yang{2}^{(n-1)}$ as well. Firstly, the image of $\yang{2}^{(n-1)}$ under the evaluation homomorphism is in the centralizer $\uni{n}^{\gl{n-2}}$ (see \cite[Cor. 1.7.2]{Molev}).

Moreover, the generators $h_{n-1}(u)$, $f_{n-1}(u)$ and $e_{n-1}(u)$ of $\yang{2}^{(n-1)}$ are in $\psi(\yang{2})$ after multiplication on some elements of $Z\yang{n-2}$:
\begin{gather*}
    h_{n-1}(u) = t_{1\ldots n-2}^{1\ldots n-2}(u+n-2)\cdot \psi(t_{11}(u)) \cdot\left(t_{1\ldots n-2}^{1\ldots n-2}(u+n-2)\right)^{-1},\\
    f_{n-1}(u) = t_{1\ldots n-2}^{1\ldots n-2}(u+n-2)\cdot\psi(t_{21}(u))\cdot\left(t_{1\ldots n-2}^{1\ldots n-2}(u+n-2)\right)^{-1},\\
    e_{n-1}(u) = \psi(t_{12}(u)).
\end{gather*} The image of the quantum minor $t_{1\ldots n-2}^{1\ldots n-2}(u+n-2)$ is in the center $Z\uni{n-2}$, hence, it acts by scalar and does not affect on reducibility.

Through the proof of Proposition \ref{gz_eigen_basis_proposition} we discussed that $\yang{2}$ is generated by the coefficients of the formal series
\begin{align*}
    t_{11}(u),\, t_{12}(u),\, t_{21}(u), \,\qdet{u}.
\end{align*} Since $K_{\mu}$ is an irreducible representation of $\yang{2}$, the coefficients of $\qdet{u}$ act by scalar. Therefore, $K_{\mu}$ is an irreducible representation of the subalgebra, generated by the coefficients of
\begin{align*}
    t_{11}(u), t_{12}(u), t_{21}(u).
\end{align*}

As we concluded, the elements $\psi(t_{11}(u))$, $\psi(t_{21}(u))$ and $\psi(t_{12}(u))$ are in the $\yang{2}^{(n-1)}$ up to multiplication to some elements of the center $Z\yang{n-2}$, which act by scalar in $K_{\mu}$. Thus, we have the next proposition.

\begin{proposition}
    $K_{\mu}$ is the irreducible $\yang{2}^{(n-1)}$-module. Consequently, $V_{\lambda}$ considered as $\yang{2}^{(n-1)}$-module is completely reducible. The same property holds for every $\yang{2}^{(i)}$.
\end{proposition}

\begin{proof}
    What is left is to explain why $V_{\lambda}$ is completely reducible for $\yang{2}^{(i)}$ with $i<n-1$. For simplicity, let us consider $\yang{2}^{(n-2)}$.

    As we discussed,
    \begin{align*}
        V_{\lambda}\big|_{\yang{n-1}} = \bigoplus_{\mu \subset \lambda}V_{\mu},
    \end{align*} where $V_{\mu}$ is the irreducible $\yang{n-1}$-module of the highest weight $\mu$ which satisfies certain properties.

    We observe that $\yang{2}^{(n-2)}\subset\yang{n-1}$, hence, it remains to prove that $V_{\mu}$ is completely reducible as $\yang{2}^{(n-1)}$-module, but it has already been proved since the last arguments do not depend on the dimension $n$.
    The same proof obviously works for every $\yang{2}^{(i)}$.
\end{proof}

We do not discuss this argument in detail because we will directly prove complete reducibility of $V_{\lambda}$ as $\yang{2}$-module later.

\subsection{Connection to Verhik-Okounkov approach} We should mention that this paper was inspired by Vershik and Okounkov's article \cite{VershikOkounkov} which is dedicated to the similar problem of representations of the symmetric group $S_n$.

In \cite{VershikOkounkov}, a representation $W$ of $\mathbb{C}[S_n]$ is implicitly viewed as a representation of the degenerate affine Hecke algebra $H(n)$ (see \cite[Remark 3.3]{VershikOkounkov}), whose analogue in this paper is $\yang{n}$. The representation is restricted to the degenerate affine Hecke subalgebra $H(2)=H^{(i)}(2)$, generated by two \textrm{YJM}-elements $X_i$ and $X_{i+1}$ and the Coxeter transposition $s_i$, whose analogues are $\yang{2}^{(i)}$. The subalgebras $\yang{2}^{(i)}$ are satisfied similar to $H(2)$ properties within the context of the search the Gelfand-Tsetlin realization. Namely $\yang{n}$ is generated by $\yang{n-1}$ and $\yang{2}^{(n-1)}$ and $\yang{2}^{(n-1)}$ commutes with $\yang{n-2}$. In terms of the branching graph, the $\yang{2}^{(n-1)}$ acts only in the $n-1$-th level.

In the preface, the authors of \cite{VershikOkounkov} declare that one of their goal is to justify the appearance of the Young diagrams in the theory of representations of $S_n$. Actually, the classification of finite-dimensional representations of $S_n$ is re-proved in \cite{VershikOkounkov}.

We lack of that problem since the semistandard Young tableaux appear in the Gelfand-Tsetlin realization at the beginning. We do not know whether the classification of irreducible finite-dimensional $\gl{n}$-modules can be naturally proved only through the Yangian theory. However, the similar techniques are justified since the realization of $\gl{n}$-modules is more complicated.

\subsection{The Gelfand-Tsetlin realization of \texorpdfstring{$\gl{n}$}{gl n}}
Recall that on account of \hyperref[Section2]{Section 2}, we fixed the Gelfand-Tsetlin basis $\{v_{\Lambda}\}$ indexed by the Gelfand-Tsetlin tableaux
\begin{center}\begin{tabular}{c }
         $\lambda_{n1}$ ~  $\lambda_{n2}$ ~ \dots ~ $\lambda_{n,n-1}$ ~$\lambda_{nn}$ \\
         $\lambda_{n-1,1}$ ~ \dots ~ $\lambda_{n-1,n-1}$ \\
         $\ddots$ ~ $\dots$ ~\reflectbox{$\ddots$} \\
         $\lambda_{21}$ ~ $\lambda_{22}$ \\
         $\lambda_{11}$
\end{tabular}
\end{center}

\begin{definition}
    Suppose that $W$ is an arbitrary $\yang{2}$-module. We call a non-zero vector $w$ the almost-highest vector if the relation
    \begin{align*}
        \frac{h_1(u)}{h_2(u)}\cdot w = \frac{P(u+1)}{P(u)}w
    \end{align*} holds for some polynomial $P(u)\in\mathbb{C}[u]$.
\end{definition}

This definition is motivated since highest vectors of irreducible representations are almost-highest. The search of all almost-highest vectors of a $\yang{2}$-module helps to find highest vectors and, hence, to decompose the module as a direct sum of the irreducible representations. Now we want to find almost-highest vectors of $V_{\lambda}$.

\begin{proposition}
\label{GF_basis_eqigenvectors_for_hi}
    The Gelfand-Tsetlin basis vectors are eigenvectors of the Drinfeld series $h_i(u)$ and
    \begin{align}
    \label{h_i_eigenvector}
        h_i(u) \cdot v_{\Lambda} = \frac{\prod_{j=1}^{i}(u+\lambda_{ij}+i-j)}{u\prod_{j=1}^{i-1}(u+\lambda_{i-1,j}+i-j)}v_{\Lambda}.
    \end{align}
\end{proposition}

\begin{proof}
    We begin by recalling the formula of $h_i(u)$:
    \begin{align*}
        h_i(u) = t_{1\ldots i}^{1\ldots i}(u+i-1)\big(t_{1\ldots i-1}^{1\ldots i-1}(u+i-1)\big).
    \end{align*} The quantum minors $t_{1\ldots i}^{1\ldots i}(u+i-1)$ and $t_{1\ldots i-1}^{1\ldots i-1}(u+i-1)$ are the shifted quantum determinants of $\yang{i}$ and $\yang{i-1}$ respectively. As $v_{\Lambda}$ is in some irreducible subrepresentations of $\gl{i}$ and $\gl{i-1}$, which are uniquely defined by its Gelfand-Tsetlin tableau, $v_{\lambda}$ is in the corresponding evaluation $\yang{i}$ and $\yang{i-1}$-modules of the highest weight.

    By Proposition \ref{action_of_quantum_determinant}, we obtain
    \begin{gather*}
        t_{1\ldots i}^{1\ldots i}(u+i-1) \cdot v_{\Lambda} = (1+\lambda_{i1}(u+i-1)^{-1}) \cdot\ldots\cdot (1+\lambda_{ii}u^{-1})v_{\Lambda},\\
        t_{1\ldots i-1}^{1\ldots i-1}(u+i-1) \cdot v_{\Lambda} = (1+\lambda_{i-1,1}(u+i-1)^{-1}) \cdot \ldots\cdot (1+\lambda_{i-1,i-1}(u+1)^{-1})v_{\Lambda}.
    \end{gather*} Since $v_{\Lambda}$ is an eigenvector of $t_{1\ldots i-1}^{1\ldots i-1}(u+i-1)$, the action of the $\big(t_{1\ldots i-1}^{1\ldots i-1}(u+i-1)\big)^{-1}$ is multiplying by the invertible eigenvalue. Now formula \eqref{h_i_eigenvector} follows immediately.
\end{proof}

The next theorem provides a description of the almost-highest vectors among the Gelfand-Tsetlin basis in $V_{\lambda}$

\begin{theorem}
\label{criterion_of_almost-highest_vector}
    A Gelfand-Tsetlin basis vector $v_{\Lambda}$ is almost-highest if and only if its Gelfand-Tsetlin tableau satisfies the condition
    \begin{align*}
        \lambda_{i+1,1}=\lambda_{i1}~\text{and}~\lambda_{ij} = \min{\{\lambda_{i+1,j}, \lambda_{i-1,j-1}\}}.
    \end{align*} for all $j\in\{2,\ldots,i\}$.
\end{theorem}

\begin{proof}
    Let us first recall that we identify $h_1(u)$ and $h_2(u)$ with $h_i(u)$ and $h_{i+1}(u)$ respectively.
    
    Applying formula \eqref{h_i_eigenvector}, we have
    \begin{align}
    \label{bad_formula_for_hi_hi1_action}
        \frac{h_i(u)}{h_{i+1}(u)} v_{\Lambda} = \frac{\prod_{j=1}^{i}(u+\lambda_{ij}+i-j)}{\prod_{j=1}^{i-1}(u+\lambda_{i-1,j}+i-j)} \cdot \frac{\prod_{j=1}^{i}(u+\lambda_{ij}+i+1-j)}{\prod_{j=1}^{i+1}(u+\lambda_{i+1,j}+i+1-j)}v_{\Lambda}.
    \end{align}

    It is a simple matter that the fraction
    \begin{align*}
        \frac{u+a}{u+b}\text{,}~ a-b\in\mathbb{Z}_{\geq 0}
    \end{align*}is equivalent to $\frac{P(u+1)}{P(u)}$, where $P(u)$ is equal to
    \begin{align*}
        (u+b)(u+b+1)\cdot\ldots\cdot(u+a-1).
    \end{align*}

    By the definition of Gelfand-Tsetlin tableau, the next relations hold:
    \begin{align}
    \label{lambda_ij_relations}
        \lambda_{ij} \geq \lambda_{i-1,j} ~\text{and}~\lambda_{ij} \geq \lambda_{i+1,j+1}.
    \end{align} After reordering the fractions of equation \eqref{bad_formula_for_hi_hi1_action}, we have
    \begin{multline}
    \label{good_formula_hi_hi1_action}
        \frac{h_i(u)}{h_{i+1}(u)} v_{\Lambda} = \frac{\prod_{j=1}^{i-1}(u+\lambda_{ij}+i-j)}{\prod_{j=1}^{i-1}(u+\lambda_{i-1,j}+i-j)} \cdot\\
        \cdot \frac{\prod_{j=1}^{i-1}(u+\lambda_{ij}+i+1-j)}{\prod_{j=1}^{i-1}(u+\lambda_{i+1,j+1}+i-j)} \cdot \frac{u+\lambda_{ii}}{u+\lambda_{i+1,i+1}} \cdot \frac{u+\lambda_{ii}+1}{u+\lambda_{i+1,1}+i}v_{\Lambda}.
    \end{multline} We conclude from the formula above and relations \eqref{lambda_ij_relations} that every fraction in the product except the last one is equivalent to $\frac{P(u+1)}{P(u)}$ for some polynomial $P(u)$ and so is their product. Our next claim is to show when the relation with fixed polynomial $P(u)$ and $a-b\in\mathbb{Z}_{\geq 0}$
    \begin{align*}
        \frac{P(u+1)}{P(u)} \cdot \frac{u+b}{u+a} = \frac{Q(u+1)}{Q(u)}
    \end{align*} holds for some polynomial $Q(u)\in\mathbb{C}[u]$.

    With the notation $F(u)=\frac{Q(u)}{P(u)}$, we have
    \begin{align*}
        \frac{u+b}{u+a} = \frac{F(u+1)}{F(u)}.
    \end{align*} There is the unique $F(u)$ up to scalar satisfying the equality above since periodic rational functions do not exist. In fact,
    \begin{align*}
        F(u) = \frac{1}{(u+b)(u+b+1)\cdot\ldots\cdot(u+a-1)}.
    \end{align*} From this we deduce that
    \begin{align*}
        Q(u) = \frac{P(u)}{(u+b)(u+b+1)\cdot\ldots\cdot(u+a-1)},
    \end{align*} and if $Q(u)$ is a polynomial, then $P(u)$ is divisible for the
    \begin{align*}
        (u+b)(u+b+1)\cdot\ldots\cdot(u+a-1)
    \end{align*} or the numbers $-b,-b-1,\ldots,-a+1$ are among of its roots, which is the same. In that case, throughout the proof, we say that a polynomial $P(u)$ defines the string $\mathit{S}(a,b)$ (see Definition \ref{string_definition}).

    By formula \eqref{good_formula_hi_hi1_action}, we see that the polynomial $P(u)$, corresponding to the product of fractions except the last one, defines the next strings:
    \begin{gather*}
        \mathcal{S}_{1} = \{\mathit{S}(\lambda_{i1}+i-1, \lambda_{i-1,1}+i-1), \ldots, \mathit{S}(\lambda_{i,i-1}+1, \lambda_{i-1,i-1}+1)\},\\
        \mathcal{S}_{2} = \{\mathit{S}(\lambda_{i1}+i+1, \lambda_{i+1,2}+i), \ldots, \mathit{S}(\lambda_{i,i-1}+2, \lambda_{i+1,i}+1)\},\\
        \mathit{S}(\lambda_{ii}, \lambda_{i+1,i+1}).
    \end{gather*} Among them, as a set, we should find the string
    \begin{gather*}
        \mathcal{S} = \mathit{S}(\lambda_{i+1,1}+i, \lambda_{ii}+1).
    \end{gather*} Now we will use relations \eqref{lambda_ij_relations} several times.

    First of all, all elements of the string $\mathit{S}(\lambda_{ii}, \lambda_{i+1,i+1})$ are strictly less than $\lambda_{ii}+1$, that is, we can exclude it.

    Secondly, notice that the strings of the $\mathcal{S}_{1}$ are ordered which means that for each allowable index $j$ all elements of the string $\mathit{S}(\lambda_{ij}+i-j, \lambda_{i-1,j}+i-j)$ are strictly greater than all elements of the next string $\mathit{S}(\lambda_{i,j+1}+i-j-1, \lambda_{i-1,j+1}+i-j-1)$. Similar property holds for $\mathcal{S}_{2}$.

    Finally, relations \eqref{lambda_ij_relations} give us
    \begin{align*}
        \lambda_{ii} \leq \lambda_{i+1,i} ~\text{and}~ \lambda_{ii} \leq \lambda_{i-1,i-1},
    \end{align*} and if both of the inequalities are strict, the first element of the searching string $\mathcal{S}$ would not be found neither in the $\mathcal{S}_{1}$ nor $\mathcal{S}_{2}$. Thus, if $v_{\Lambda}$ is almost-highest, then 
    \begin{align*}
        \lambda_{ii} = \lambda_{i+1,i} ~\text{or}~ \lambda_{ii} = \lambda_{i-1,i-1}.
    \end{align*}

    If $\lambda_{ii}=\lambda_{i-1,i-1}$, then the last string of the $\mathcal{S}_{1}$ is reduced, and the searching string $\mathcal{S}$ after the reduction starts with $\lambda_{i,i-1}+1$. This element is only equal to the greatest element of the last string of the $\mathcal{S}_{2}$, and after reduction the searching string starts with $\lambda_{i,i-1}+2$ and all elements of the last string of $\mathcal{S}_{2}$ are less than it, that is, the last string can be excluded.

    If $\lambda_{ii} = \lambda_{i+1,i}$, likewise, the last string of $\mathcal{S}_{2}$ is reduced, the searching string immediately starts with $\lambda_{i,i-1}+2$ and all elements of the last string of $\mathcal{S}_{1}$ are less than it, that is, the last string can be excluded.

    The rest of the proof runs as before. The smallest element of the searching string is less than or equal to the smallest elements of the last strings of $\mathcal{S}_{1}$ or $\mathcal{S}_{2}$, and it is equal to only if 
    \begin{align*}
        \lambda_{ij} = \lambda_{i+1,j} ~\text{or}~ \lambda_{ij} = \lambda_{i-1,j-1},
    \end{align*} which is desired conclusion because of relations \eqref{lambda_ij_relations}.
\end{proof}

\begin{corollary}
\label{h_ih_i1_eigenvalue}
    If a Gelfand-Tsetlin basis vector $v_{\Lambda}$ is almost-highest, then
    \begin{align*}
        \frac{h_i(u)}{h_{i+1}(u)} \cdot v_{\Lambda} = \frac{(u+\lambda_{i+1,1}+i-1)\cdot\prod_{j=2}^{i}(u+\min{\{\lambda_{i+1,j}, \lambda_{i-1,j-1}\}}+i-j)}{\prod_{j=1}^{i-1}(u+\max{\{\lambda_{i+1,j+1},\lambda_{i-1,j}\}}+i-j) \cdot (u+\lambda_{i+1,i+1})}v_{\Lambda}.
    \end{align*}
\end{corollary}

\begin{proof}
    Since the relations from Theorem \ref{criterion_of_almost-highest_vector} holds,
    \begin{multline*}
        h_i(u) \cdot v_{\Lambda} = (u+\lambda_{i+1,1}+i-1)\cdot\prod_{j=2}^{i}(u+\min{\{\lambda_{i+1,j}, \lambda_{i-1,j-1}\}}+i-j) \cdot\\
        \cdot u^{-1} \cdot\prod_{j=1}^{i-1}(u+\lambda_{i-1,j}+i-j)^{-1}.
    \end{multline*} Let $J\subset\{2,\ldots,i\}$ stand for the subset of indices such that
    \begin{align*}
        \lambda_{i+1,j} \geq \lambda_{i-1,j-1} \Rightarrow \lambda_{ij} = \lambda_{i-1,j-1}.
    \end{align*} By formula \eqref{h_i_eigenvector} for $h_{i+1}(u)$, we see that if $j\notin J$ then the factor
    \begin{align*}
        (u+\lambda_{i+1,j}+i+1-j)
    \end{align*} is reduced from the numerator and denominator of the eigenvalue of $h_{i+1}(u)$. Therefore, it follows that
    \begin{align*}
        h_{i+1}(u) \cdot v_{\Lambda} = \frac{u+\lambda_{i+1,i+1}}{u} \prod_{j\in J}\frac{(u+\lambda_{i+1,j}+(i+1)-j)}{(u+\lambda_{i-1,j-1}+i-(j-1))}.
    \end{align*} It suffices to observe that the denominator of the eigenvalue of $\frac{h_{i}(u)}{h_{i+1}(u)}$ for $v_{\Lambda}$ contains of
    \begin{gather*}
        (u+\lambda_{i+1,j}+(i+1)-j),~\text{if}~j\in J \iff \lambda_{i+1,j} \geq \lambda_{i-1,j-1},\\
        (u+\lambda_{i-1,j-1}+i-(j-1)),~\text{if}~j\notin J \iff \lambda_{i-1,j-1}\geq\lambda_{i+1,j}
    \end{gather*} and the factor $(u+\lambda_{i+1,i+1})$. So, the proposition follows.
\end{proof}

It remains to prove that the found almost-highest vectors are actually the highest. To this end, we shall discuss Gelfand-Tsetlin basis.

\begin{proposition}
\label{h_i_defines_GZ_basis}
    A Gelfand-Tsetlin basis vector $v_{\Lambda}$ is uniquely defined up to scalar by the eigenvalues of $h_1(u),\ldots,h_n(u)$. Moreover, the first $h_1(u),\ldots,h_i(u)$ define first $i$ rows of the corresponding tableau $\Lambda$ inclusively and the last $h_{i+1}(u),\ldots,h_n(u)$ define the rows from the $i$-th inclusively to the last one.
\end{proposition}

\begin{proof}
    The first statement is a direct consequence of Theorem \ref{gelfand_tsetlin_theorem} but as we promised we check it straightforwardly to prove the second statement. 

    We now use formula \eqref{h_i_eigenvector} and relations \eqref{lambda_ij_relations}. The series $h_1(u)$ acts by multiplying the vector $v_{\Lambda}$ by 
    \begin{align*}
        \frac{u+\lambda_{11}}{u},
    \end{align*} thus to find the $\lambda_{11}$ we should multiply the eigenvalue by $u$ and get the only root of the resulting polynomial.

    Then, as far as we know $\lambda_{11}$ and 
    \begin{align*}
        h_2(u)\cdot v_{\Lambda} = \frac{(u+\lambda_{21}+1)(u+\lambda_{22})}{u(u+\lambda_{11}+1)}v_{\Lambda},
    \end{align*} we can multiply the eigenvalue of $h_2(u)$ by $u(u+\lambda_{11}+1)$ and 
    \begin{align*}
        \big(\lambda_{21}+1, \lambda_{22}\big)
    \end{align*} are to be the ordered roots of the resulting polynomial. That is correct since the roots are integer.

    Continuing, as far as we know all rows from the first to the $i-1$-th, we can multiply the eigenvalue of $h_{i}(u)$ by
    \begin{align*}
        (u+\lambda_{i-1,1}+i-1) \cdot \ldots\cdot (u+\lambda_{i-1,i-1}+1)u,
    \end{align*} and $\lambda_{ij}+i-j$ is equal to the $j$-th ordered root of the corresponding polynomial.

    Actually, the similar algorithm can be started with the last row because we know the highest weight of the representation. We can now proceed analogously. If we multiply the eigenvalue of $h_{n}^{-1}(u)$ to the rational function
    \begin{align*}
        \frac{1}{u}(u+\lambda_{n1}+n-1) \cdot\ldots\cdot(u+\lambda_{nn}),
    \end{align*} we get the polynomial which the $j$-th ordered root is equal to $\lambda_{n-1,j}+n-j$. The same argument remains valid, and the proof is complete.
\end{proof}

\begin{definition}
    For any fixed Gelfand-Tsetlin tableau $\Lambda$ a subspace $\mathcal{A}(\Lambda)$ is spanned by Gelfand-Tsetlin basis vectors $v_{\Lambda^\prime}$ where all rows of the tableau $\Lambda^\prime$ except the $i$-th are equal to the corresponding rows of the tableau $\Lambda$.
\end{definition}

\begin{proposition}
\label{yang_holds_lambda}
    The subspace $\mathcal{A}(\Lambda)$ is invariant under the $\yang{2}^{(i)}$-action.
\end{proposition}

\begin{proof}
    By  relation \eqref{drinfeld_commutations} for the Drinfeld generators, the elements of $\yang{2}^{(i)}$ commute with $h_j(u)$ unless $j=i$ or $j=i+1$. As a result, for every $w\in\yang{2}^{(i)}$ we have
    \begin{align*}
        h_j(u) \cdot w = \mu_j(u) w, ~j\not=i,i+1,
    \end{align*} where $\mu_j(u)$ is an eigenvalue of $h_j(u)$ for $v_{\Lambda}$. By Proposition \ref{h_i_defines_GZ_basis}, all rows except the $i$-th one of the tableaux of the Gelfand-Tsetlin basis vectors which generate $\mathcal{A}(\Lambda)$ are uniquely determined, which is our assertion.
\end{proof}

\begin{proposition}
\label{the_only_almost-highest_vector}
    There is only one up to scalar non-zero almost-highest vector in $\mathcal{A}(\Lambda)$.
\end{proposition}

\begin{proof}
    The almost-highest vectors are eigenvectors for both $h_i(u)$ and $h_{i+1}(u)$, thus, they are from the Gelfand-Tsetlin basis. By Theorem \ref{criterion_of_almost-highest_vector}, the $i$-th row of the almost-highest vector can be defined by the $i-1$ and the $i+1$-th rows of the tableau, but all Gelfand-Tsetlin basis vectors from $\mathcal{A}(\Lambda)$ have the same $i-1$ and $i+1$-th rows of the tableau. Thus, the required almost-highest vector is uniquely defined up to scalar.
\end{proof}

\begin{theorem}
    The $\gl{n}$-module $V_{\lambda}$ considered as a representation of $\yang{2}^{(i)}$ is a direct sum of irreducible $\yang{2}$-modules. 
\end{theorem}

\begin{proof}
    As a vector space, $V_{\lambda}$ is a direct sum
    \begin{align*}
        V_{\lambda} = \bigoplus_{\Lambda} \mathcal{A}(\Lambda),
    \end{align*} with $\Lambda$ are all Gelfand-Tsetlin tableaux which differs not only in the $i$-th row. By the last Proposition \ref{the_only_almost-highest_vector}, there is only one almost-highest vector in $\mathcal{A}(\Lambda)$, let it be $v_{\Lambda}$. By Proposition \ref{yang_holds_lambda},
    \begin{align*}
        \yang{2}^{(i)} \cdot v_{\lambda} \subseteq \mathcal{A}(\Lambda).
    \end{align*}

    Consider the representation $\yang{2}^{(i)} \cdot v_{\lambda}$ and its irreducible subrepresentation. As we have already mentioned, the highest vector is always almost-highest. Since the $v_{\lambda}$ is the only almost-highest vector in $\mathcal{A}(\Lambda)$ and hence in $\yang{2}^{(i)} \cdot v_{\lambda}$, the vector $v_{\lambda}$ is the highest vector of that irreducible subrepresentation.

    The proof is completed by showing that the dimension of $\yang{2}^{(i)} \cdot v_{\lambda}$  equals the dimension of $\mathcal{A}(\Lambda)$, that is, the number of all Gelfand-Tsetlin tableaux with fixed all rows except the $i$-th.

    Indeed, by Corollary \ref{h_ih_i1_eigenvalue}, after some automorphism of multiplication that irreducible subrepresentation is isomorphic to
    \begin{align*}
        \Lira{1} \otimes\ldots\otimes \Lira{i},
    \end{align*} where for all $j\in\{1,\ldots,i\}$
    \begin{gather*}
        \alpha_j = \min{\{\lambda_{i+1,j},\lambda_{i-1,j-1}\}}+i-j,\\
        \beta_j = \max{\{\lambda_{i+1,j+1},\lambda_{i-1,j}\}}+i-j
    \end{gather*} with $\alpha_1=\lambda_{i+1,1}+i-1$ and $\beta_{i}=\lambda_{i+1,i+1}$. The dimension of each factor of the tensor product is equal to
    \begin{align*}
        \min{\{\lambda_{i+1,j},\lambda_{i-1,j-1}\}} - \max{\{\lambda_{i+1,j+1},\lambda_{i-1,j}\}}+1.
    \end{align*} That is the number of elements in
    \begin{align*}
        \mathit{S}(\min{\{\lambda_{i+1,j},\lambda_{i-1,j-1}\}}+1, \max{\{\lambda_{i+1,j+1},\lambda_{i-1,j}\}}),
    \end{align*} but $\lambda_{ij}$ can be equal only to an element of that string since relations \eqref{lambda_ij_relations} hold.

    Therefore, the direct sum decomposition
    \begin{align*}
        V_{\lambda} = \bigoplus_{\Lambda} \mathcal{A}(\Lambda)
    \end{align*} is the direct sum of the irreducible subrepresentations of $\yang{2}^{(i)}$.
\end{proof}

\begin{corollary}
\label{effectively_describing_subrepr}
    The $\yang{2}^{(i)}$-module $\mathcal{A}(\Lambda)$ is isomorphic to
    \begin{align*}
        V(\Lambda) = \Lira{1} \otimes\ldots\otimes \Lira{i},
    \end{align*} where for all $j\in\{1,\ldots,i\}$
    \begin{gather*}
        \alpha_j = \min{\{\lambda_{i+1,j},\lambda_{i-1,j-1}\}}+i-j,\\
        \beta_j = \max{\{\lambda_{i+1,j+1},\lambda_{i-1,j}\}}+i-j
    \end{gather*} with $\alpha_1=\lambda_{i+1,1}+i-1$ and $\beta_{i}=\lambda_{i+1,i+1}$ after the automorphism of multiplication by
    \begin{align*}
        f_{\Lambda}(u)=u^{1-i} \prod_{j=1}^{i-1}(u+\lambda_{i-1,j}+i-j).
    \end{align*} We emphasize that each pair of the corresponding strings $\mathcal{S}(\alpha_i,\beta_i)$ does not intersect thus the formulas from Theorem \ref{yangian_2_GF_realization} hold.
\end{corollary}

\begin{proof}
    We only need to prove the formula for $f_{\Lambda}(u)$. It uniquely determines by
    \begin{align*}
        f_{\Lambda}(u)\lambda_1(u) = \lambda_1^\prime(u),
    \end{align*} with $\lambda_1(u)$ and $\lambda_1^\prime(u)$ are equal to eigenvalues of the highest vectors of $\mathcal{A}(\Lambda)$ and $V(\Lambda)$ respectively. The formula for $\lambda_1(u)$ can be proved by Proposition \ref{GF_basis_eqigenvectors_for_hi} following the proof of Corollary \ref{h_ih_i1_eigenvalue} and by formula \eqref{T_11_formula}
    \begin{align*}
        \lambda_1^\prime(u) = u^{-i}(u+\lambda_{i+1,1}+i-1)\cdot\prod_{j=2}^{i}(u+\min{\{\lambda_{i+1,j}, \lambda_{i-1,j-1}\}}+i-j).
    \end{align*}
\end{proof}

Now we are ready to write the explicit realization of the $V_{\lambda}$. 
\begin{theorem}[Gelfand-Tsetlin realization]
\label{Gelfand_Tsetlin_realization}
    Suppose that $V_{\lambda}$ is the finite-dimensional irreducible representation of the highest weight $\lambda$. Then the Gelfand-Tsetlin basis vectors $\{v_{\Lambda}\}$ can be chosen from the corresponding lines in a way that the next formulas hold:
    \begin{gather*}
        E_{ii}\cdot v_{\Lambda} = \left(\sum_{j=1}^{i}\lambda_{ij}-\sum_{j=1}^{i-1}\lambda_{i-1,j}\right)v_{\Lambda},\\
        E_{i+1,i}\cdot v_{\Lambda} = \sum_{j=1}^{i}\frac{1}{(l_{i1}-l_{ij})\cdot\ldots\land^{j}\ldots\cdot(l_{ii}-l_{ij})}v_{\Lambda-\delta_{ij}},\\
        E_{i,i+1}\cdot v_{\Lambda} = -\sum_{j=1}^{i} \frac{\prod_{k=1}^{i+1}(l_{i+1,k}-l_{ij}) \cdot \prod_{k=1}^{i-1}(l_{i-1,k}-l_{ij}-1)}{(l_{i1}-l_{ij})\cdot\ldots\land^{j}\ldots\cdot(l_{ii}-l_{ij})}v_{\Lambda+\delta_{ij}},
    \end{gather*} where $l_{ij}=\lambda_{ij}-j$, the tableau $\Lambda\pm\delta_{ij}$ differs from $\Lambda$ only by $\lambda_{ij}\pm 1$ and equals zero if $\Lambda\pm\delta_{ij}$ is not a Gelfand-Tsetlin tableau, and $\land^{j}$ means a skip of the $j$-th factor equal zero.

    Moreover, that formulas uniquely determine the representation.
\end{theorem}

\begin{proof}
    Recall formulas \eqref{row_of_h_i}--\eqref{row_of_e_i} of the first terms of the Drinfeld generators.
    By formula \eqref{h_i_eigenvector} we immediately establish that
    \begin{align*}
        E_{ii}\cdot v_{\Lambda} = \left(\sum_{j=1}^{i}\lambda_{ij}-\sum_{j=1}^{i-1}\lambda_{i-1,j}\right)v_{\Lambda}.
    \end{align*}

    The elements $E_{i+1,i}$ and $E_{i,i+1}$ are the first non-zero terms of the formal series $f_i(u)$ and $e_i(u)$ respectively as well. We wish to apply the formulas from Corollary \ref{sl_n_drinfeld_formulas_yang2}, and as we discussed, they only depend on the tensor product $V(\Lambda)$, but not on $f_{\Lambda}(u)$.

    Hence we may apply formulas \eqref{f_1_whole_formula_with_u} and \eqref{e_1_whole_formula_with_u} directly. The only thing we need to know is 
    $\alpha_j-r_j$. We observe that for any vector $v_{\Lambda}$
    \begin{align*}
        f_{\Lambda}(u)h_i(u) \cdot v_{\Lambda} = \prod_{j=1}^{i}(u+\lambda_{ij}+i-j).
    \end{align*} By Remark \ref{uniqeuly_defined_T11}, we deduce that
    \begin{align*}
        \alpha_j-r_j = \lambda_{ij}+i-j=l_{ij}+i
    \end{align*} From now on, one may assume that each Gelfand-Tsetlin basis vector $v_{\Lambda}$ from $\mathcal{A}(\Lambda)$ is chosen in a way that it is equal to the basis vector of the irreducible representation $V(\Lambda)$ defined in the \ref{v_r_definition}. We immediately get
    \begin{align*}
        E_{i+1,i}\cdot v_{\Lambda} = \sum_{j=1}^{i}\frac{1}{(l_{i1}-l_{ij})\cdot\ldots\land^{j}\ldots\cdot(l_{ii}-l_{ij})}v_{\Lambda-\delta_{ij}}.
    \end{align*} It is important to pay attention that adding $\pm1$ to indices of the Gelfand-Tsetlin basis in $\yang{2}$ is equivalent to adding $\mp1$ to the $ij$-th element of the corresponding Gelfand-Tsetlin tableau according to the formula of $r_j$.

    To establish the last formula we have to calculate
    \begin{align*}
        \prod_{m=1}^{k}(\alpha_m-\alpha_j+r_j)(\beta_m-\alpha_j+r_j-1)
    \end{align*} in terms of the Gelfand-Tsetlin tableau $\Lambda$. On account of Corollary \ref{effectively_describing_subrepr}, all elements $\alpha_j,\beta_j$ for all $j\in\{1,\ldots,i\}$ correspond to all elements of the $i-1$-th and the $i+1$-th row of the $\Lambda$.

    More precisely, if element $\alpha_j$ or $\beta_j$ corresponds to the $i+1$-th row then it is equal to $l_{i+1,j}+i$, else it is equal to $l_{i-1,j}+i-1$. Thus after the reordering we get
    \begin{align*}
        \prod_{m=1}^{k}(\alpha_m-\alpha_j+r_j)(\beta_m-\alpha_j+r_j-1) = \prod_{k=1}^{i+1}(l_{i+1,k}-l_{ij}) \cdot \prod_{k=1}^{i-1}(l_{i-1,k}-l_{ij}-1)
    \end{align*} and the last formula for $E_{i,i+1}$ is established.

    In fact, we get the formulas for each index $i$ separately because we have chosen $v_{\Lambda}$ to be equal to the basis vectors $v_r$ in $\yang{2}^{(i)}$. But it is mild since for different indices $i$ and $i^\prime$ the Gelfand-Tsetlin basis vectors of $\mathcal{A}^{(i)}(\Lambda)$ are each in different $\mathcal{A}^{(i^\prime)}(\Lambda)$. Hence after the choice of basis vectors for $i$ they can be multiplied by fixed scalar for each $\mathcal{A}^{(i)}(\Lambda)$ to be the basis vector for $\mathcal{A}^{(i^\prime)}(\Lambda)$, and all of the formulas hold.

    Finally, the elements $E_{ii}$, $E_{i+1,i}$ and $E_{i,i+1}$ are standard generators of $\gl{n}$, and the found formulas uniquely determine the corresponding irreducible representation.
\end{proof}

\begin{remark}
    Our formulas is slightly different from the similar formulas of \cite[Th. 5.4.1]{Molev}, indeed, there one can find
    \begin{gather*}
        E_{i+1,i}\cdot v_{\Lambda} = \sum_{j=1}^{i}\frac{\prod_{k=1}^{i-1}(l_{i-1,k}-l_{ij})}{(l_{i1}-l_{ij})\cdot\ldots\land^{j}\ldots\cdot(l_{ii}-l_{ij})}v_{\Lambda-\delta_{ij}},\\
        E_{i,i+1}\cdot v_{\Lambda} = -\sum_{j=1}^{i} \frac{\prod_{k=1}^{i+1}(l_{i+1,k}-l_{ij})}{(l_{i1}-l_{ij})\cdot\ldots\land^{j}\ldots\cdot(l_{ii}-l_{ij})}v_{\Lambda+\delta_{ij}}.
    \end{gather*} We should observe that our formulas and above are equivalent, that is, the linear automorphism of $V_{\Lambda}$ exists which transform our formulas to those.
\end{remark}

\begin{proof}
    Of course, the linear automorphism pretends to be diagonal:
    \begin{align*}
        v_{\Lambda}\to C_{\Lambda}v_{\Lambda}.
    \end{align*} So, let us arise the automorphism by the next rule:
    \begin{itemize}
        \item if $\Lambda$ is the tableau of highest vector, then $C_{\Lambda}=1$,
        \item for all allowable tableaux the formula $C_{\Lambda} = \prod_{k=1}^{i-1}(l_{i-1,k}-l_{ij})C_{\Lambda-\delta_{ij}}$ holds.
    \end{itemize} It is easy to check that the product
    \begin{align*}
        \prod_{k=1}^{i-1}(l_{i-1,k}-l_{ij})
    \end{align*} is equal to zero only in a case $\lambda_{i-1,j}=\lambda_{ij}$. But in that case the tableau $C_{\Lambda-\delta_{ij}}$ is not allowable.

    Since every allowable Gelfand-Tsetlin tableau can be gotten from the tableau of the highest vector by subtracting $\delta_{ij}$ several times we define a linear automorphism. By definition, it is satisfying
    \begin{align*}
        \frac{C_{\Lambda}}{C_{\Lambda-\delta_{ij}}} = \prod_{k=1}^{i-1}(l_{i-1,k}-l_{ij}) \iff \frac{C_{\Lambda+\delta_{ij}}}{C_{\Lambda}} = \prod_{k=1}^{i-1}(l_{i-1,k}-(l_{ij}+1)) ,
    \end{align*} now one can easily check that under this linear automorphism our formulas go to the formulas of \cite{Molev}.
\end{proof}